\documentclass{article}

\usepackage{xspace}

\usepackage{rotating}
\usepackage{graphicx}
\usepackage{tabularx}
\usepackage{amsfonts}
\usepackage{amsmath}
\usepackage{amsthm}% or letter or a5paper or ... etc
\usepackage{mathabx}
\usepackage{proof}
\usepackage{mathdots}
\usepackage{xcolor}
\usepackage{amssymb}
\usepackage{color}
\usepackage{mathrsfs}
\usepackage[all]{xy}
\usepackage{lscape}
\usepackage{stmaryrd}
\usepackage{bussproofs}

\usepackage{tikz}
\usetikzlibrary{positioning,shapes,shapes.multipart}
\usetikzlibrary{arrows,snakes,backgrounds}

\newtheorem{theorem}{Theorem}[section]  
\newtheorem{corollary}[theorem]{Corollary}
\newtheorem{lemma}[theorem]{Lemma}
\newtheorem{proposition}[theorem]{Proposition}

\newtheorem{definition}[theorem]{Definition}
\newtheorem{remark}[theorem]{Remark}

\newcommand{\B}[1]{\mathbf{#1}}
\newcommand{\D}[1]{\mathscr{#1}}

\newcommand{\C}[1]{\mathcal{#1}}

\newcommand{\OV}[1]{\mathbin{\overline{#1}}}

 \newcommand{\Nd}{\texttt{NI}$^2$\xspace}
 \newcommand{\Nv}{{\texttt{NI}$^{\mkern-3mu\mathord{\vee}\mkern-3mu}$}\xspace}
 \newcommand{\Ndv}{\texttt{NI}$^{2\mathord{\vee}}$\xspace}
 \newcommand{\Nda}{\texttt{NI}$^2_{at}$\xspace}
 \newcommand{\Ld}{$\mathcal{L}^2$\xspace}
 \newcommand{\Lv}{$\mathcal{L}^{\vee}$\xspace}
 \newcommand{\Ldv}{$\mathcal{L}^{2\vee}$\xspace}

 \newcommand{\Ldm}{\mathcal{L}^2}
 
 \newcommand{\Ldvm}{\mathcal{L}^{2\vee}}

\usepackage{varwidth}
\newsavebox{\mypti}
\newsavebox{\mypto}
\newsavebox{\mypta}
\newsavebox{\myptu}
\newsavebox{\mypte}

 \newcommand*\SD[3]{%
\begin{tikzpicture}[baseline=(C.base)]
    \node[rectangle split, rectangle split horizontal, rectangle split parts=2, rectangle split part align=base, inner sep=1pt](C) {${#2}$ \nodepart{two}  $\left\ldbrack{#3}\right\rdbrack$};
  \end{tikzpicture}
}

\newcommand*\SF[3]{
\begin{tikzpicture}[baseline=(C.base)]
    \node[rectangle split, rectangle split horizontal, rectangle split parts=2, rectangle split part fill={black!10,black!10} , rectangle split part align=base, inner sep=1pt](C) {${#2}$ \nodepart{two}  \colorbox{white}{$ {#3} $}};
  \end{tikzpicture}
}

\newcommand*\SDI[3]{
\begin{tikzpicture}[baseline=(C.base)]
    \node[rectangle split, rectangle split horizontal, rectangle split parts=2, rectangle split part align=base, inner xsep=1pt, inner ysep=-1pt](C) {${#2}$ \nodepart{two}  \colorbox{white}{\makebox[1.2em][c]{$\left\ldbrack{#3}  \right\rdbrack\;$}}};
  \end{tikzpicture}
}

\newcommand*\SFI[3]{
\begin{tikzpicture}[baseline=(C.base)]
    \node[rectangle split, rectangle split horizontal, rectangle split parts=2, rectangle split part fill={black!10,black!10} , rectangle split part align=base, inner sep=1pt](C) {${#2}$ \nodepart{two}  \colorbox{white}{\makebox[1.2em][c]{$ {#3} $}}};
  \end{tikzpicture}
}

\newcommand*\SDO[3]{
\begin{tikzpicture}[baseline=(C.base)]
    \node[rectangle split, rectangle split horizontal, rectangle split parts=2, rectangle split part align=base, inner xsep=1pt, inner ysep=-1pt](C) {${#2}$ \nodepart{two}  \colorbox{white}{\makebox[.8em][c]{$\left\ldbrack {#3} \right\rdbrack\;$}}};
  \end{tikzpicture}
}

\newcommand*\SFO[3]{
\begin{tikzpicture}[baseline=(C.base)]
    \node[rectangle split, rectangle split horizontal, rectangle split parts=2, rectangle split part fill={black!10,black!10} , rectangle split part align=base, inner sep=1pt](C) {${#2}$ \nodepart{two}  \colorbox{white}{\makebox[.8em][c]{$ {#3} $}}};
  \end{tikzpicture}
}

\newcommand*\cexpD[3]{
\begin{tikzpicture}[baseline=(C.base)]
    \node[rectangle split, rectangle split horizontal, rectangle split parts=2, rectangle split part fill={black!10,black!10} , inner sep=1pt
    ](C) {${#2}$ \nodepart{two}  \colorbox{white}{\makebox[1.2em][c]{$ {#3} $}}};
  \end{tikzpicture}
}

\makeatletter
\newcommand*{\textlabel}[2]{%
  \edef\@currentlabel{#1}% Set target label
  \phantomsection% Correct hyper reference link
  #1\label{#2}% Print and store label
}
\makeatother

\usepackage{authblk}

\title{The naturality of natural deduction}
\author[1]{Luca Tranchini}
\author[2]{Paolo Pistone}
\author[3]{Mattia Petrolo}
\affil[1,2]{\small\emph{Wilhelm-Schickard-Institut, Universit\"at T\"ubingen}}
\affil[3]{\small\emph{Centro de Ci\^encias Naturais e Humanas, Universidade Federal do ABC}}

\date{}

\EnableBpAbbreviations

\begin{document}
\maketitle

%\setcounter{page}{1}     %%  Initial page number. Do not change. 

%%%%%%  FORMAT OF AUTHOR AND TITLE INFORMATION:
 
%\threeAuthorsTitle{L.\ts Tranchini}{P.\ts Pistone}{M. 
%\ts Petrolo}{The naturality\newline of natural deduction}
%

%%%%%%  INFORMATION FOR FOOTER OF FIRST PAGE
%%%%%%  will be inserted by the Editorial Office.

%\PresentedReceived{Name of Editor}{December 1, 2005}

%%%%%%  ABSTRACT AND KEYWORDS  (obligatory)

\begin{abstract}Developing a suggestion by Russell, Prawitz showed how the usual natural deduction inference rules for disjunction, conjunction and absurdity can be derived using those for implication and the second order quantifier in propositional intuitionistic second order logic \Nd. It is however well known that the translation does not preserve the relations of identity among derivations induced by the permutative conversions and immediate expansions for the definable connectives, at least when the equational theory of \Nd is assumed to consist only of $\beta$ and $\eta$ equations. On the basis of the categorial interpretation of \Nd, we introduce a new class of equations expressing what in categorial terms is a naturality condition satisfied by the transformations interpreting \Nd-derivations. We show that the Russell-Prawitz translation does preserve identity of proof with respect to the enriched system by highlighting the fact that naturality corresponds to a generalized permutation principle. Finally we sketch how these results could be used to investigate the properties of connectives definable in the framework of higher-level rules.
\end{abstract}

%\Keywords{Identity of proof, permutative conversions, dinaturality condition, functorial interpretation,  $\eta$-conversion, Russell-Prawitz translation, second order logic.}

%%%%%  THE BODY OF THE PAPER

\section*{Introduction}\label{intro} 

Since Russell \cite{Russell1903} it is known that in propositional second order logic it is possible to define disjunction (as well as conjunction and absurdity) using implication and the second order quantifier.
Prawitz \cite{Prawitz1965} showed how the usual natural deduction inference rules for disjunction (as well as for conjunction and absurdity) can be derived using those for implication and the second order quantifier in propositional intuitionistic second order logic. Following \cite{Ferreira2009,Ferreira2013}, but see also  \cite{Acz01}, we refer to the derivability-preserving translation of propositional intuitionistic second order logic \Ndv{} into its disjunction- (and conjunction- and absurdity-) free fragment \Nd{} as the ``Russell-Prawitz translation''.

The conversions used to establish the normalization results for natural deduction systems can be viewed as inducing an equational theory on derivations \cite{Prawitz1971}. One may therefore expect that the Russell-Prawitz translation preserves the relations of equivalence among derivations. However, this is not the case in general, at least when one considers the usual equational theory for \Nd, i.e.~the one consisting of the so called $\beta$- and $\eta$-equations. In particular, although $\beta$-equivalent derivations in the full language are mapped onto $\beta$-equivalent derivations in the implicational fragment, the same does not happen for $\eta$-equivalent derivations, nor for derivations which are equivalent modulo the equations corresponding to the permutative (or commutative) conversions for disjunction (we will refer to these as \mbox{$\gamma$-equations}).

In categorial interpretations of propositional intuitionistic second order logic \cite{Girard1992}, formulas are interpreted as particular functors, and  derivations are viewed as \emph{natural transformations} between these functors. The naturality of the transformations can be expressed as a particular class of equations. Generalizing these equations results in an extension of the $\beta\eta$-equational theory for \Nd{}, and the Russell-Prawitz translation maps $\gamma$-equations onto a particular sub-class of these new equations. 

With the goal of making these results accessible to a wider community, we will give a purely proof-theoretical presentations of them. In particular, we will highlight the fact that the naturality condition, in terms of natural deduction, corresponds to a general permutation principle.

We will show that our results are actually a generalization of some elementary facts which have gone so far unnoticed, namely that particular classes of instances of the $\gamma$- and $\eta$-equations in \Ndv{} are mapped by the Russell-Prawitz translation onto equations which are included in the $\beta\eta$-equational theory of \Nd.

In a recent series of papers Ferreira and Ferreira \cite{Ferreira2009, Ferreira2013} advanced another approach to solve the problem of preserving $\eta$- and $\gamma$-equations in \Ndv{} (and its extensions with the other definable connectives). The main ingredient of their approach is a result (that they call {\em instantiation overflow}) which holds for the fragment \Nda{} of \Nd{} enjoying the sub-formula property obtained by restricting the elimination rule for the second order quantifier $\forall$E to atomic substitution: in \Nda an unrestricted applications of $\forall$E is derivable provided that the premise of the rule application has the form of the Russell-Prawitz translation of some propositional formula. In a sequel to this paper,  we will analyse Ferreira and Ferreira's results within the framework here introduced.

\section{Preliminaries}\label{sec-prel}

Given a countable set of propositional variables $\mathcal{V}$, the formulas $\Phi^{2\mkern-3mu\vee\mkern-3mu}$ of the language \Ldv{} are defined by the following grammar:
\begingroup\makeatletter\def\f@size{10}\check@mathfonts
$$\Phi^{2\mkern-3mu\vee\mkern-3mu}::=  \quad \mathcal{V} \quad |  \quad (\forall \mathcal{V}\Phi^{2\mkern-3mu\vee\mkern-3mu})\quad | \quad (\Phi^{2\mkern-3mu\vee\mkern-3mu}\supset\Phi^{2\mkern-3mu\vee\mkern-3mu}) \quad | \quad (\Phi^{2\mkern-3mu\vee\mkern-3mu}\vee\Phi^{2\mkern-3mu\vee\mkern-3mu}).$$\endgroup
As usual, we omit outermost parentheses and we take iterated implications to associate to the right.  We call  \Lv{} the restriction of \Ldv{} to the $\{\supset,\vee\}$ language fragment and \Ld{} the restriction of \Ldv{} to the $\{\supset,\forall\}$ language fragment. By $A\ldbrack B/X\rdbrack$ we indicate the result of substituting the formula $B$ for the variable $X$ in  $A$.

We define the natural deduction system \Ndv{} as follows:\footnote{In rule schemata we indicate in square brackets the assumptions which can be discharged by rule applications, whereas in derivation schemata we use square brackets only to indicate an arbitrary number of occurrences of a formula, if the formula is in assumption position, or of the whole sub-derivation having the formula in brackets as conclusion. In derivation schemata, we indicate discharge with roman italics letters (possibly with subscripts) placed above the discharged assumptions and near the rule label. Although according to the definition of derivation (which follows strictly the one of \cite{TS96}), every assumption (both discharged and undischarged) carries a label, to improve readability we will follow the common convention of omitting the labels from undischarged assumptions.}

\begin{definition}[\Ndv-derivation]\label{ni2}$\phantom{A}$
\begin{itemize}
 \item For all formulas $A$ of \Ldv and natural number $n$, $\stackrel{n}{A}$ is a \Ndv-derivation of $A$ from undischarged assumption $A$;
 \item if $\mathscr{D}^*$, $\mathscr{D}'$, $\mathscr{D}''$, $\mathscr{D}_1$ and $\mathscr{D}_2$ are \Ndv-derivations then the following:
\begingroup\makeatletter\def\f@size{10}\check@mathfonts
$$\begin{array}{cc}
\AXC{$\mathscr{D}^*$}\noLine\UIC{$A%\ldbrack Y/X\rdbrack
$}\RightLabel{\footnotesize $\forall$I}
\UIC{$\forall X A$}\DP\qquad\phantom{AAAAA}

&  
\AXC{$\mathscr{D}'$}\noLine\UIC{$\forall X A$}
\RightLabel{\footnotesize $\forall$E}
\UIC{$A \ldbrack B/X\rdbrack$}
\DP
\\
& \\

\AXC{$\stackrel{n}{[A]}$}
\noLine\UIC{$\mathscr{D}_1$}\noLine\UIC{$B$}
\RightLabel{\footnotesize $\supset$I $(n)$}
\UIC{$A\supset B$}
\DP \quad\quad\quad\quad\ 
& 

\AXC{$\mathscr{D}'$}\noLine\UIC{$A\supset B$}
\AXC{$\mathscr{D}''$}\noLine\UIC{$A$}
\RightLabel{\footnotesize $\supset$E}
\BIC{$B$}
\DP
\\
& \\

\AXC{$\D D'$}\noLine\UIC{$A$}
\RightLabel{\footnotesize $\vee$I$_1$}
\UIC{$A \vee B$}
\DP 
\quad
\AXC{$\D D'$}\noLine\UIC{$B$}
\RightLabel{\footnotesize $\vee$I$_2$}
\UIC{$A \vee B$}
\DP \quad\qquad\quad\qquad

&

\AXC{$\D D'$}\noLine \UIC{$A\vee B$}
		    \AXC{$\stackrel{n}{[A]}$}
		    \noLine
		    \UIC{$\D D_{1}$}
		    \noLine
		    \UIC{$C$} 
				      \AXC{$\stackrel{m}{[B]}$}
				      \noLine
				       \UIC{$\D D_{2}$}
					    \noLine
				      \UIC{$C$}
\RightLabel{\footnotesize $\vee$E $(n,m)$}
\TIC{$C$}
\DP
\end{array}$$\endgroup
are \Ndv-derivations as well (provided $X$ does not occur free in the assumptions of $\mathscr{D}^*$), where the assumptions $\stackrel{n}{A}$ (resp.~$\stackrel{m}{B}$) that are undischarged in $\mathscr{D}_1$ (resp.~in $\mathscr{D}_2$) become discharged in the derivation of $A\supset B$ (resp.~of $C$).

\item Nothing else in an \Ndv-derivation.
 
\end{itemize}

\end{definition}

The natural deduction system \Nv{} for propositional intuitionistic logic is the restriction of \Ndv{} to  the language \Lv{}.
The natural deduction system \Nd{} for second order intuitionistic logic (which corresponds via Curry-Howard to the polymorphic lambda calculus of Girard and Reynolds) is the restriction of \Ndv{} to  the language \Ld{}.

% % \begin{remark}\label{sub2}By $\mathscr{D}\ldbrack B/X\rdbrack$  we indicate the result of substituting the formula $B$ for the variable $X$ in $\mathscr{D}$. For a precise definition see \cite{xx}. We only remind that
% % 
% % \begin{itemize}
% %  \item If $\D D = \AXC{$\mathscr{D}_1$}\noLine\UIC{$A$}\RightLabel{\footnotesize $\forall$I}\UIC{$\forall X A$}\DP$ then for all $B$, $\D D \llbracket B/X\rrbracket = \D D$
% % \end{itemize}
% % \end{remark}

Adopting the terminology of lambda calculus, we will refer to the instances of the equation schemata in table \ref{conversions} as $\beta$- and $\eta$-equations (we indicate by $\mathscr{D}\ldbrack B/X\rdbrack$ the result of substituting the formula $B$ for the variable $X$ in $\mathscr{D}$).

\begin{table}[b!]\caption{$\beta$- and $\eta$-conversions}\label{conversions}

\rule{\textwidth}{.5pt}

\begingroup\makeatletter\def\f@size{9}\check@mathfonts
\begin{tabularx}{\textwidth}{XX}
%\hline
{\begin{equation*}\tag{$\beta{\supset}$}
\label{betaimp}\def\defaultHypSeparation{\hskip .02in}
\AXC{$\stackrel{n}{[A]}$}\noLine\UIC{$\mathscr{D}$}\noLine\UIC{$B$}\RightLabel{\scriptsize $\supset$I $(n)$}\UIC{$A\supset  B$}\AXC{$\mathscr{D}'$}\noLine\UIC{$A$}\RightLabel{\scriptsize$\supset$E}\BIC{$B$}\DP   =_{\beta}^{\supset}\ \AXC{$\mathscr{D}'$}\noLine\UIC{$[A]$}\noLine\UIC{$\mathscr{D}$}\noLine\UIC{$B$}\DP
\end{equation*}}
&
{\begin{equation*}\tag{$\eta{\supset}$}
\label{etaimp}\def\defaultHypSeparation{\hskip .02in}
 \AXC{$\mathscr{D}$}\noLine\UIC{$A\supset B$}\AXC{$\stackrel{n}{A}$}\RightLabel{\scriptsize$\supset$E}\BIC{$B$}\RightLabel{\scriptsize$\supset$I $(n)$}\UIC{$A\supset B$}\DP\ =_{\eta}^{\supset} \ \AXC{$\mathscr{D}$}\noLine\UIC{$A\supset B$}\DP
 \end{equation*}
$$\text{(provided $n$ does not occur in $\D D$)}$$

}
\\
&\\
{\begin{equation*}\tag{$\beta{\forall}$}
\label{beta2}
\AXC{$\mathscr{D}$}\noLine\UIC{$A%\ldbrack Y/X\rdbrack
$}\RightLabel{\footnotesize$\forall$I}\UIC{$\forall X A$}\RightLabel{\footnotesize$\forall$E}\UIC{$A \ldbrack B/X\rdbrack$}\DP  =_{\beta}^2 \AXC{$\mathscr{D}\ldbrack B/X\rdbrack$}\noLine\UIC{$A\ldbrack B/X\rdbrack$}\DP\end{equation*}}&
{\begin{equation*}\tag{$\eta{\forall}$}
\label{eta2}
\AXC{$\mathscr{D}$}\noLine\UIC{$\forall X A$}\RightLabel{\footnotesize$\forall$E}\UIC{$A$}\RightLabel{\footnotesize$\forall$I}\UIC{$\forall X A$}\DP \ =_{\eta}^2\ \AXC{$\mathscr{D}$}\noLine\UIC{$\forall X A$}\DP
\end{equation*}}
\\
&\\
\multicolumn{2}{l}{{\parbox{\textwidth}{\begin{equation*}
\tag{$\beta{\lor}$}
\label{betaor}\def\defaultHypSeparation{\hskip .02in}
\AXC{$\mathscr{D}$}\noLine\UIC{$A_i$}\RightLabel{\footnotesize $\vee$I$_i$}\UIC{$A_1\vee A_2$}\AXC{$\stackrel{n}{[A_1]}$}\noLine\UIC{$\mathscr{D}_1$}\noLine\UIC{$C$}\AXC{$\stackrel{m}{[A_2]}$}\noLine\UIC{$\mathscr{D}_2$}\noLine\UIC{$C$}\RightLabel{\footnotesize $\vee$E $(n,m)$}\TIC{$C$}\DP  \ =_{\beta}^{\vee}\ \AXC{$\mathscr{D}$}\noLine\UIC{$[A_i]$}\noLine\UIC{$\mathscr{D}_i$}\noLine\UIC{$C$}\DP
\end{equation*}}}}\\
&\\
\multicolumn{2}{r}{\parbox{\textwidth}{\begin{equation*}\tag{$\eta\lor$}
\label{etaor}\def\defaultHypSeparation{\hskip .05in}
\AXC{$\mathscr{D}$}\noLine\UIC{$A\vee B$}\AXC{$\stackrel{n}{A}$}\RightLabel{\footnotesize $\vee$I}\UIC{$A\vee B$}\AXC{$\stackrel{m}{B}$}\RightLabel{\footnotesize $\vee$I}\UIC{$A\vee B$}\RightLabel{\footnotesize $\vee$E $(n,m)$}\TIC{$A\vee B$}\DP\ =_{\eta}^{\vee} \ \AXC{$\mathscr{D}$}\noLine\UIC{$A\vee B$}\DP\end{equation*}
$$\text{(provided $n, 	m$ do not occur in $\D D$)}$$%
}}\\
\end{tabularx}\endgroup

\medskip
\rule{\textwidth}{.5pt}
\end{table}

\medskip
As usual, we will refer to the rewriting rules obtained  by orienting these equations from left to right  as $\beta$-, $\eta$-{\em reductions} and to those obtained by orienting these equations from right to left as $\beta$-, $\eta$-{\em expansions}. The deductive patterns displayed on the left-hand side (respectively right-hand side) of these equations will be referred to as $\beta$-, $\eta$-{\em redexes} (resp. {\em reducts}). A derivation is $\beta$-, $\eta$-normal if and only if it contains no $\beta$-, $\eta$-redex. 

As is well known (see, e.g.,\cite{GLT89}, p.~76), in order for normal derivations in \Nv{} to enjoy the sub-formula property, a further kind of equations (which we will call $\gamma$) need to be assumed (we indicate with $\dagger$E the application of an elimination rule for ``some'' connective $\dagger$ and with $\overline{\mathscr{D}}$ the derivations of its minor premises):
\begingroup\makeatletter\def\f@size{10}\check@mathfonts
\begin{equation*}\tag{$\gamma\lor$}\label{gammaor}
\def\defaultHypSeparation{\hskip .0in}\AXC{$\mathscr{D}$}\noLine\UIC{$A\vee B$}\AXC{$\stackrel{n}{[A]}$}\noLine\UIC{$\mathscr{D}_1$}\noLine\UIC{$C$}\AXC{$\stackrel{m}{[B]}$}\noLine\UIC{$\mathscr{D}_2$}\noLine\UIC{$C$}\RightLabel{\footnotesize  $\lor$E $(n,m)$}\TIC{$C$}\AXC{$\overline{\mathscr{D}}$}\RightLabel{\footnotesize$\dagger$E}\BIC{$D$}\!\!\!\!\!\!\DP =_{\gamma}^{\vee} \def\defaultHypSeparation{\hskip .0in}\AXC{$\mathscr{D}$}\noLine\UIC{$A\vee B$}\AXC{$\stackrel{n}{[A]}$}\noLine\UIC{$\mathscr{D}_1$}\noLine\UIC{$C$}\AXC{$\overline{\mathscr{D}}$}\RightLabel{\footnotesize$\dagger$E}\BIC{$D$}\AXC{$\stackrel{m}{[B]}$}\noLine\UIC{$\mathscr{D}_2$}\noLine\UIC{$C$}\AXC{$\overline{\mathscr{D}}$}\RightLabel{\footnotesize$\dagger$E}\BIC{$D$}\RightLabel{\footnotesize $\vee$E $(n,m)$}\TIC{$D$}\DP
\end{equation*}\endgroup

The left to right orientation of $\gamma$-equations are usually called {\em permutations} (sometimes also commutations) rather than reductions. However, we will speak of $\gamma$-redexes and $\gamma$-normal derivations. We will use {\em conversion} to cover reductions, expansions and permutations.

In categorial logic (see, e.g, \cite{Seely1979}), as well as in the literature on typed lambda calculi with sums (see for instance \cite{Lindley2007}), however, one usually considers a more general equation schema, namely this:
\begingroup\makeatletter\def\f@size{10}\check@mathfonts
\begin{equation*}\tag{$\gamma_g\!\lor$}\label{gammaorG}
\AxiomC{$\D D$}
\noLine
\UnaryInfC{$A\lor B$}
\AxiomC{$\stackrel{n}{[A]}$}
\noLine
\UnaryInfC{$\D D_{1}$}
\noLine
\UnaryInfC{$C$}
\AxiomC{$\stackrel{m}{[B]}$}
\noLine
\UnaryInfC{$\D D_{2}$}
\noLine
\UnaryInfC{$C$}
\RightLabel{\footnotesize$\lor$E $(n,m)$}
\TrinaryInfC{$[C]$}
\noLine
\UnaryInfC{$\D D_3$}
\noLine
\UnaryInfC{$ D$}
\DisplayProof
\ \ =_{\gamma_g}^{\lor} \!\!\!
\AxiomC{$\D D$}
\noLine
\UnaryInfC{$A\lor B$}
\AxiomC{$\stackrel{n}{[A]}$}
\noLine
\UnaryInfC{$\D D_{1}$}
\noLine
\UnaryInfC{$[C]$}
\noLine
\UnaryInfC{$\D D_3$}
\noLine
\UnaryInfC{$ D$}
\AxiomC{$\stackrel{m}{[B]}$}
\noLine
\UnaryInfC{$\D D_{2}$}
\noLine
\UnaryInfC{$[C]$}
\noLine
\UnaryInfC{$\D D_3$}
\noLine
\UnaryInfC{$ D$}
\RightLabel{\footnotesize$\lor$E $(n,m)$}
\TrinaryInfC{$D$}
\DisplayProof
\end{equation*}
\endgroup
of which $\gamma$ is just an instance obtained by taking $\D D_3$ to consist of  the application of $\dagger$E alone. We call {\em generalized permutations} the left-to-right orientation of these equations.

We indicate by $=_{\beta\eta\gamma_{(g)}}^{\supset2\vee}$ the smallest congruent equivalence realation on \Ndv-derivations induced by these equations (by removing one or more of the subscripts or superscripts we indicate the opportune restrictions of this equivalence).

%is c The proof of proposition \ref{conv} suggest to generalize the permutative conversion for disjunction \eqref{gammaor} to \eqref{gammaorG}, that 

It is worth observing that the schema \eqref{gammaorG} is more general than \eqref{gammaor} in two respects: first, it allows the downwards permutation of an application of $\vee$E across more than one rule application at once; second, it allows the downward permutations of $\vee$E not only when its conclusion is the major premise of an elimination rule, but also when it is the premise of an introduction or the minor premise of an elimination. In fact, the equational theory induced by \eqref{gammaor} is strictly contained in the one induced by \eqref{gammaorG}:

\begin{proposition}[$=_{\gamma}^{\lor}\subsetneq =_{\gamma_g}^{\lor}$] There are derivations $\D D_1$ and $\D D_2$ such that $\D D_1 =_{\gamma_g}^{\lor} \D D_2$ and $\D D_1\neq_{\gamma}^{\lor}\D D_2$.
\end{proposition}

\begin{proof} The proposition is established by remarking that the rewriting system consisting of the $\gamma$-conversion is strongly normalizing \cite{Prawitz1971}, while the one consisting of the $\gamma_g$-conversion is not \cite{Ghani1995,Altenkirch2001,Lindley2007}.
\end{proof}

The notion of generalized permutation applies to all connective with elimination rules in general form (\cite{SH1984}). If we consider the extension of \Lv and \Nv with the $\bot$ and its elimination rule $\AXC{$\bot$}\RightLabel{\footnotesize$\bot$E}\UIC{$C$}$, the standard permutation for $\bot$ is \cite{GLT89}:
\begingroup\makeatletter\def\f@size{10}\check@mathfonts
\begin{equation*}\tag{$\gamma\bot$}\label{gammabot}\AXC{$\mathscr{D}$}\noLine\UIC{$\bot$}\RightLabel{\footnotesize$\bot$E}\UIC{$C$}\AXC{$\overline{\mathscr{D}}$}\RightLabel{\footnotesize$\dagger$E}\BIC{$D$}\DP\quad =_\gamma^{\bot}\quad \AXC{$\mathscr{D}$}\noLine\UIC{$\bot$}\RightLabel{\footnotesize$\bot$E}\UIC{$D$}\DP\end{equation*}
\endgroup
and it generalizes to the following conversion:
\begingroup\makeatletter\def\f@size{10}\check@mathfonts
\begin{equation*}\tag{$\gamma_g\bot$}\label{gammabotG}\AXC{$\mathscr{D}$}\noLine\UIC{$\bot$}\RightLabel{\footnotesize$\bot$E}\UIC{$[C]$}\noLine\UIC{$\mathscr{D}'$}\noLine\UIC{$D$}\DP\quad  =_{\gamma_g}^{\bot}\quad \AXC{$\mathscr{D}$}\noLine\UIC{$\bot$}\RightLabel{\footnotesize$\bot$E}\UIC{$D$}\DP\end{equation*}
\endgroup 
expressing in categorial terms the initiality of $\bot$.

\section{Properties of the Russell-Prawitz translation}\label{sec-rp}
Prawitz \cite{Prawitz1965} showed how to extend Russell's translation of formulas of \Ldv into formulas of \Ld  into a translation of \Ndv\ into derivations of \Nd. The Russell-Prawitz translation (for short, RP-translation), is defined as follows:\footnote{We use $\equiv$ for syntactical identity.}

\begin{definition}[Russell-Prawitz translation: $\Ldvm \mapsto \Ldm$]\label{rp-formula}
The {\em Russell-Prawitz translation} (henceforth {\em RP-translation}) of an \Ldv-formula $A$ is the \Ld-formula $A^*$ defined by induction on the number of logical signs in $A$ as follows:
\begingroup\makeatletter\def\f@size{10}\check@mathfonts
\begin{equation*}
\begin{array}{r@{\ \equiv\ }l}
Y^{*}& Y   \\   (A\supset B)^{*} & A^{*}\supset B^{*} \\
\big (\forall Y A\big )^{*} & \forall Y  A^{*} \\   (A\vee B)^{*}& \forall X((A^{*}\supset X)\supset (B^{*}\supset X)\supset X)
\end{array}
\end{equation*}\endgroup

\end{definition}

\begin{definition}[Russell-Prawitz translation: $\text{\Ndv} \mapsto \text{\Nd}$]\label{rp-derivation}The {\em RP-trans-lation} of an \Ndv-derivation $\mathscr{D}$ is the \Nd-derivation $\mathscr{D}^*$ defined by induction on the number of inference rules applied in $\mathscr{D}$ as follows:

\begin{itemize}
\item if $\mathscr{D}\equiv\ \stackrel{n}{A}$, then $\mathscr{D}^*\equiv\ \stackrel{n}{A^*}$;

\item if $\D D \equiv {\small \AxiomC{$\D D'$}
\noLine
\UnaryInfC{$A$}
\RightLabel{\footnotesize$\lor$I$_{1}$}
\UnaryInfC{$A\lor B$}
\DisplayProof}$, then $\D D^{*}\equiv
{\small \AxiomC{$\stackrel{n}{A^{*}\supset X}$}
\AxiomC{$\D D'^{*}$}
\noLine
\UnaryInfC{$A^{*}$}
\RightLabel{\footnotesize$\supset$E}
\BinaryInfC{$X$}
\RightLabel{\footnotesize$\supset$I}
\UnaryInfC{$(B^{*}\supset X)\supset X$}
\RightLabel{\footnotesize$\supset$I $(n)$}
\UnaryInfC{$(A^{*}\supset X)\supset (B^{*}\supset X)\supset X$}
\RightLabel{\footnotesize$\forall I$}
\UnaryInfC{$(A\lor B)^{*}$}
\DisplayProof}
$ 

(provided $n$ does not occur in $\mathscr{D'}^*$);
\item if $\D D\equiv
{\small \AxiomC{$\D D'$}
\noLine
\UnaryInfC{$B$}
\RightLabel{\footnotesize$\lor I_{2}$}
\UnaryInfC{$A\lor B$}
\DisplayProof}$, then $\D D^{*}\equiv
{\small \AxiomC{$\stackrel{n}{B^{*}\supset X}$}
\AxiomC{$\D D'^{*}$}
\noLine
\UnaryInfC{$B^{*}$}
\RightLabel{\footnotesize$\supset$E}
\BinaryInfC{$X$}
\RightLabel{\footnotesize$\supset$I $(n)$}
\UnaryInfC{$(B^{*}\supset X)\supset X$}
\RightLabel{\footnotesize$\supset$I}
\UnaryInfC{$(A^{*}\supset X)\supset (B^{*}\supset X)\supset X$}
\RightLabel{\footnotesize$\forall I$}
\UnaryInfC{$(A\lor B)^{*}$}
\DisplayProof}
$ 

(provided $n$ does not occur in $\mathscr{D'}^*$);
\item if $\D D\equiv
{\small \AxiomC{$\D D'$}
\noLine
\UnaryInfC{$A\lor B$}
\AxiomC{$\stackrel{n}{[A]}$}
\noLine
\UnaryInfC{$\D D_{1}$}
\noLine
\UnaryInfC{$C$}
\AxiomC{$\stackrel{m}{[B]}$}
\noLine
\UnaryInfC{$\D D_{2}$}
\noLine
\UnaryInfC{$C$}
\RightLabel{\footnotesize$\lor$E $(n,m)$}
\TrinaryInfC{$C$}
\DisplayProof}$, then
\begingroup\makeatletter\def\f@size{10}\check@mathfonts$$ \D D^{*}\equiv
\def\defaultHypSeparation{\hskip .09in}\def\ScoreOverhang{2pt}\AxiomC{$\D D'^{*}$}
\noLine
\UnaryInfC{$(A\lor B)^{*}$}
\RightLabel{\footnotesize$\forall E$}
\UnaryInfC{$(A^*\supset C^*)\supset (B^*\supset C^*)\supset C^*$}
\AxiomC{$\stackrel{n}{[A^{*}]}$}
\noLine
\UnaryInfC{$\D D_{1}^{*}$}
\noLine
\UnaryInfC{$C^{*}$}
\RightLabel{\footnotesize$\supset$I $(n)$}
\UnaryInfC{$A^*\supset C^*$}
\RightLabel{\footnotesize$\supset$E}
\BinaryInfC{$(B^*\supset C^*)\supset C^*$}
\AxiomC{$\stackrel{m}{[B^{*}]}$}
\noLine
\UnaryInfC{$\D D_{2}^{*}$}
\noLine
\UnaryInfC{$C^{*}$}
\RightLabel{\footnotesize$\supset$I $(m)$}
\UnaryInfC{$B^*\supset C^*$}
\RightLabel{\footnotesize$\supset$E}
\BinaryInfC{$C^*$}
\DisplayProof
$$\endgroup
\item all other rules are translated in a trivial way.

\end{itemize}
\end{definition}

The RP-translation $^*$ maps $\beta$-equivalent derivations into $\beta$-equivalent derivations in \Nd:

\begin{proposition}[$=_{\beta}^{2\supset\vee}\ \stackrel{*}{\mapsto}\ =_{\beta}^{2\supset}$]\label{betapres}
If $\mathscr{D}_1 =_\beta^{\supset2\vee} \mathscr{D}_2$ then $\mathscr{D}_1^* =_\beta^{\supset2} \mathscr{D}_2^*$.
\end{proposition}
\begin{proof}
It suffices to show this in the case of a $\vee$-redex. If the redex contains an application of $\vee$I$_{1}$, then we just verify that:
\begingroup\makeatletter\def\f@size{10}\check@mathfonts
\begin{equation*}
\AxiomC{$\stackrel{m}{A^{*}\supset X}$}
\AxiomC{$\D D^{*}$}
\noLine
\UnaryInfC{$A^{*}$}
\RightLabel{\footnotesize$\supset$E}
\BinaryInfC{$X$}
\RightLabel{\footnotesize$\supset$I}
\UnaryInfC{$(B^{*}\supset X)\supset X$}
\RightLabel{\footnotesize$\supset$I $(m)$}
\UnaryInfC{$(A^{*}\supset X)\supset (B^{*}\supset X)\supset X$}
\RightLabel{\footnotesize$\forall I$}
\UnaryInfC{$(A\lor B)^{*}$}
\RightLabel{\footnotesize$\forall E$}
\UnaryInfC{$(A^*\supset C^*)\supset (B^*\supset C^*)\supset C^*$}
\AxiomC{$\stackrel{n_1}{[A^{*}]}$}
\noLine
\UnaryInfC{$\D D_{1}^{*}$}
\noLine
\UnaryInfC{$C^{*}$}
\RightLabel{\footnotesize$\supset$I $(n_1)$}
\UnaryInfC{$A^*\supset C^*$}
\RightLabel{\footnotesize$\supset$E}
\BinaryInfC{$(B^*\supset C^*)\supset C^*$}
\AxiomC{$\stackrel{n_2}{[B^{*}]}$}
\noLine
\UnaryInfC{$\D D_{2}^{*}$}
\noLine
\UnaryInfC{$C^{*}$}
\RightLabel{\footnotesize$\supset$I $(n_2)$}
\UnaryInfC{$B^*\supset C^*$}
\RightLabel{\footnotesize$\supset$E}
\BinaryInfC{$C^*$}
\DisplayProof\end{equation*}\endgroup
is $\beta$-equivalent to 
\begingroup\makeatletter\def\f@size{10}\check@mathfonts$$\AXC{$\D D^{*}$}
\noLine
\UIC{$[A^*]$}
\noLine
\UIC{$\D D_{1}^{*}$}
\noLine
\UIC{$C^*$}
\DisplayProof
$$\endgroup
The case where the redex contains an application of  $\vee$I$_{2}$ is treated similarly.
\end{proof}

Analogous results do {\em not} hold for the equivalence relations induced by $\eta$- and $\gamma$-equations (see e.g.~\cite{GLT89}, p.~85): in particular, it is not the case that whenever $\mathscr{D}_1 =_{\eta}^{\vee} \mathscr{D}_2$, then $\mathscr{D}_1^* =_{\eta}^{\supset2} \mathscr{D}_2^*$, actually not even $\mathscr{D}_1^* =_{\beta\eta}^{\supset2} \mathscr{D}_2^*$. Similarly, it is not the case that whenever  $\mathscr{D}_1 =_{\gamma}^{\vee}\mathscr{D}_2$, then $\mathscr{D}_1^* =_{\beta\eta}^{\supset2} \mathscr{D}_2^*$. To see this, it is enough to consider any instance of \eqref{etaor} and \eqref{gammaorG} in which the derivation $\D D$ of the major premise of the application of $\vee$E constituting the redex consists of the sole assumption of $A\vee B$.

In spite of this, it is possible to show that particular classes of instances of \eqref{etaor} and \eqref{gammaorG} are preserved by the RP-translation, in particular, those instances in which the derivation of the major premise of the application of $\vee$E displayed in the equation schemata is closed (we will refer to such instances as {\em m-closed}). The equivalence relations induced by these instances of \eqref{etaor} and \eqref{gammaorG} will be indicated with $=_{\eta^{mc}}^{\vee}$ and $=_{\gamma_g^{mc}}^{\vee}$.

To prove this fact we rely on a restricted form of normalization for \Ndv\ (namely, that any derivation $\mathscr{D}$ can be $\beta$-reduced to a $\beta$-normal derivation $\tilde{\mathscr{D}}$) and on the following:

\begin{proposition}\label{closenormintro} All closed $\beta$-normal \Ndv-derivations end with an application of an introduction rule.
 \end{proposition}

\begin{proof}The proof is by induction on the number of applications of elimination rules in a closed $\beta$-normal derivation $\mathscr{D}$.
\end{proof}

A consequence of this is that the m-closed instances of \eqref{gammaorG} and \eqref{etaor} are contained in the equational theory $=_{\beta}^{2\supset\vee}$: 

\begin{proposition}
\label{gammaclose}
If $\mathscr{D}'=_{\gamma_g^{mc}}^{\vee}\mathscr{D}''$ then $\mathscr{D}'=_{\beta}^{2\supset\vee}\mathscr{D}''$

\end{proposition}
\begin{proof}
Consider an instance of \eqref{gammaorG} in which the derivation ${\D D}$ of $A\vee B$ is {\em closed}. Call $\D D'$ and $\mathscr{D}''$ the left-hand side and right-hand side of this instance of \eqref{gammaorG}. We show that $\mathscr{D}'=^{\supset2\vee}_{\beta} \mathscr{D}''$.

Since $\D D$ is closed, it can be $\beta$-reduced into a $\beta$-normal derivation $\D D^{\sharp}$ which (by proposition \ref{closenormintro}) consists of a derivation ${\D D^\sharp}'$ of either $A$ or $B$ to which an application of one of the introduction rules is appended. If we assume that  the rule applied is $\vee$I$_1$ (the alternative case is similar), the two members of the $\gamma_g$-equation $\beta$-reduce (respectively) to the following derivations:
\begingroup\makeatletter\def\f@size{10}\check@mathfonts$$\AXC{${\mathscr{D}^\sharp}'$}\noLine\UIC{$A$}\RightLabel{\footnotesize$\vee $I$_{1}$}\UIC{$A\vee B$}\AXC{$\stackrel{n}{[A]}$}\noLine\UIC{$\mathscr{D}_1$}\noLine\UIC{$C$}\AXC{$\stackrel{m}{[B]}$}\noLine\UIC{$\mathscr{D}_2$}\noLine\UIC{$C$}\RightLabel{\footnotesize $\vee$E $(n,m)$}\TIC{$[C]$}\noLine
\UnaryInfC{$\D D_3$}
\noLine
\UnaryInfC{$ D$}
\DP \qquad \AXC{${\D D^\sharp}'$}\noLine\UIC{$A$}\RightLabel{\footnotesize$\vee$I$_{1}$}\UIC{$A\vee B$}\AXC{$\stackrel{n}{[A]}$}\noLine\UIC{$\mathscr{D}_1$}\noLine\UIC{$[C]$}\noLine
\UnaryInfC{$\D D_3$}
\noLine
\UnaryInfC{$ D$}\AXC{$\stackrel{m}{[B]}$}\noLine\UIC{$\mathscr{D}_2$}\noLine\UIC{$[C]$}\noLine
\UnaryInfC{$\D D_3$}
\noLine
\UnaryInfC{$ D$}
\RightLabel{\footnotesize $\vee$E $(n,m)$}\TIC{$D$}\DP$$\endgroup
which are clearly $\beta$-equivalent.
\end{proof}

\begin{proposition}
\label{etaclose}
If $\mathscr{D}'=_{\eta^{mc}}^{\vee}\mathscr{D}''$ then $\mathscr{D}'=_{\beta}^{2\supset\vee}\mathscr{D}''$
\end{proposition}
\begin{proof}
Consider an instance of \eqref{etaor} in which ${\D D}$ is {\em closed}. Call $\D D'$ and $\mathscr{D}''$ the left-hand side and right-hand side of this instance of \eqref{etaor}. We show $\mathscr{D}'=^{2\supset\vee}_{\beta} \mathscr{D}''$.

As in the proof of the previous proposition, the two members of the $\eta$-equation $\beta$-reduce (respectively) to the following derivations:
\begingroup\makeatletter\def\f@size{10}\check@mathfonts
\begin{equation*}\def\defaultHypSeparation{\hskip .1in}
\AXC{${\mathscr{D}^\sharp}'$}\noLine\UIC{$A$}\RightLabel{\footnotesize$\vee$I$_1$}\UIC{$A\vee B$}\AXC{$\stackrel{n}{A}$}\RightLabel{\footnotesize $\vee$I}\UIC{$A\vee B$}\AXC{$\stackrel{m}{B}$}\RightLabel{\footnotesize $\vee$I}\UIC{$A\vee B$}\RightLabel{\footnotesize $\vee$E $(n,m)$}\TIC{$A\vee B$}\DP\ \qquad\qquad \ \AXC{${\mathscr{D}^\sharp}'$}\noLine\UIC{$A$}\RightLabel{\footnotesize$\vee$I$_1$}\UIC{$A\vee B$}\DP\end{equation*}
\endgroup
which are clearly $\beta$-equivalent.
\end{proof}

\begin{remark}\label{eta2imp-pres}
Propositions analogous to \ref{gammaclose} and \ref{etaclose} hold for the m-closed instances of \eqref{eta2} and \eqref{etaimp} as well. 
\end{remark}

As by Propositions \ref{gammaclose}  and  \ref{etaclose} the m-closed instances of \eqref{gammaorG} and of \eqref{etaor} are included in the equational theory $=_{\beta}^{2\supset\vee}$, and moreover 
by Proposition \ref{betapres} the RP-translation maps $\beta^{2\supset\vee}$-equivalent derivations in \Ndv{} into $\beta^{2\supset}$-equivalent derivations in \Nd,  we have the following:

\begin{corollary}[$=_{\gamma_g^{mc}}^{\vee}\ \stackrel{*}{\mapsto}\ =_{\beta}^{2\supset}$]\label{mc-gammapres}
If $\mathscr{D}_1 =_{\gamma_g^{mc}}^{2\supset\vee} \mathscr{D}_2$ then $\mathscr{D}_1^* =_\beta^{2\supset} \mathscr{D}_2^*$.
\end{corollary}

\begin{corollary}[$=_{\eta^{mc}}^{\vee}\ \stackrel{*}{\mapsto}\ =_{\beta}^{2\supset}$]\label{mc-etapres}
If $\mathscr{D}_1 =_{\eta^{mc}}^{2\supset\vee} \mathscr{D}_2$ then $\mathscr{D}_1^* =_\beta^{2\supset} \mathscr{D}_2^*$.
\end{corollary}

\begin{remark}
A proposition analogous to \ref{gammaclose} and \ref{etaclose} and a corollary analogous to \ref{mc-gammapres} and \ref{mc-etapres} can be established for the so-called simplification conversions \cite[\S II.3.3.2.1]{Prawitz1971}, which are the left-to-right orientations of the instances of the following equation schema:
\begingroup\makeatletter\def\f@size{10}\check@mathfonts
\begin{equation*}\tag{$\sigma\vee$}
\AXC{$\mathscr{D}$}\noLine\UIC{$A_1\vee A_2$}\AXC{$\mathscr{D}_1$}\noLine\UIC{$C$}\AXC{$\mathscr{D}_2$}\noLine\UIC{$C$}\RL{\footnotesize $\vee$E}\TIC{$C$}\DP\quad =_\sigma\quad \mathscr{D}_i
\end{equation*}
$$\text{provided the displayed application of $\vee$E discharges no assumption in $\mathscr{D}_i$}$$
\endgroup
\end{remark}

In the next sections we will show that the corollaries \ref{mc-etapres} and \ref{mc-gammapres} are instances of a more general phenomenon, namely the \emph{naturality} of natural deduction derivations.

\section{The naturality of \Nd-derivations}\label{sec-nat}

In this section we introduce a naturality condition for \Nd-derivations, well-known from categorial approaches to logic, in purely proof-theoretic terms. 

Besides relying on the notion of substitution of a formula $A$ for a variable $X$ both within a formula $C$, $C\ldbrack A/X\rdbrack$, and within a derivation $\mathscr{D}$, $\mathscr{D}\ldbrack A/X\rdbrack$, we will need a further notion in order to present naturality in proof-theoretic terms: when $C$ is a positive (respectively negative) formula, $X$ is a variable and $\mathscr{D}$ is a derivation of $B$ from (undischarged) assumptions $A, \Delta$, the {\em $C$-expansion of $\mathscr{D}$ relative to $X$}, will be defined as a particular derivation, to be indicated with $\SF{X}{C}{\mathscr{D}}$, of  $C\ldbrack B/X\rdbrack$ from $C\ldbrack A/X\rdbrack$ and $\Delta$ (resp.~of $C\ldbrack A/X\rdbrack$ from $C\ldbrack B/X\rdbrack$ and $\Delta$).

In section~\ref{spx-ss}, after addressing a few remarks on substitution, we introduce the notion of $X$-safety and of positive and negative formulas and derivations that we will use in the definition of $C$-expansion. We devote section~\ref{cexp-ss} to the latter notion and its properties. Finally we introduce the naturality condition in section~\ref{nat-ss}.

Unless explicitly stated, for the rest of section~\ref{sec-nat} when we speak of derivations we mean \Nd-derivations. 

\subsection{$X$-safety, positive and negative formulas and derivations}\label{spx-ss}
For simplicity, we will identify formulas and derivations which can be obtained from each other by renaming the bound variables. Hence, 

\begin{remark}\label{alfa}We assume the following:
\begin{enumerate}
 \item Substitution is always defined, though, in some cases, the substitution $(\forall Y F) \llbracket A/X\rrbracket= \forall Y' (F'\llbracket A/X\rrbracket)$ might require a renaming of the bound variables of $F$.
\item Given two derivations $\D D$ and $\D D'$, no bound variables in $\D D$  occurs free in $\D D'$.
\end{enumerate}
\end{remark}

These assumptions will be needed in the proof of the main theorem of this section. This, as well as most results in this section below, will be shown for particular classes of derivation which are defined relative to the choice of a particular variable:

\begin{definition}[X-safety]An \Nd-derivation $\D D$ is \emph{$X$-safe}, if, for all applications {\small\AXC{$\forall Y A$}\RightLabel{\footnotesize$\forall$E}\UIC{$A\llbracket B/Y\rrbracket$}\DP} of $\forall$E in $\D D$, $X$ does not occur free in $B$.
\end{definition}

\begin{remark} $X$-safety is preserved by $\beta$-reduction, i.e.~if $\mathscr{D}$ $\beta$-reduces to $\D D'$ and $\D D$ is $X$-safe, so is $\D D'$. 
\end{remark}

\begin{remark}\label{epsfree}
For any \Ndv-derivation $\D D$  there is always a variable $X$ such that the RP-translation $\D D^{*}$ of $\D D$ is $X$-safe.
\end{remark}

Before introducing $C$-expansions, we recall the definition of positive and negative formulas and extended it to derivations. 

\begin{definition}[Positive and negative formulas and derivations]\label{sp-qsp}Given a variable $X$, a formula $C$ is {\em positive in $X$} ({\em negative in $X$}), abbreviated p-$X$ (n-$X$) iff:
\begin{itemize}
\item $C\equiv Z$, where, possibly, $Z=X$ (resp.~$C\equiv Z$ with $Z\neq X$);
\item $C\equiv A\supset B$, provided $A$ is n-$X$ (resp.~p-$X$) and $B$ is p-$X$ (resp.~n-$X$);
\item $C\equiv \forall Y A$, provided $A\llbracket Z/Y\rrbracket $ is p-$X$ (resp.~n-$X$) for $Z\nequiv X$.
\end{itemize}

A formula is said pn-$X$ if it is either p-$X$ or n-$X$.
A derivation $\mathscr{D}$ of $C$ from $\Gamma$ is p-$X$ (resp.~n-$X$, pn-$X$) iff all formulas in $\Gamma$ and $C$ are p-$X$ (resp.~n-$X$, pn-$X$). 
\end{definition}

In order to prove the main theorem of this section we need the following proposition, whose proof is given in appendix \ref{appendino}:

\begin{proposition}\label{normalspx}
If $\mathscr{D}$ is $X$-safe, $\beta$-normal and pn-$X$, all formulas occurring in $\mathscr{D}$ are either p-$X$ or n-$X$. 
\end{proposition}

\subsection{The $C$-expansion of a derivation}\label{cexp-ss}

Fixed a variable $X$, we now show how, for any p-$X$ formula $C$, a derivation of $B$ from $A,\Delta$ can be ``expanded'' into a derivation of $C\llbracket B/X\rrbracket$  from $C\llbracket A/X\rrbracket,\Delta$, to be called the $C$-expansion of $\D D$ relative to $X$, which we indicate as $\SF{X}{C}{\mathscr{D}}$ (thereby leaving the fixed variable $X$ implicit).

Similarly, for any n-$X$ formula $D$, the $D$-expansion of a derivation $\D D$ of $B$ from $A$ will be a derivation of $D\llbracket A/X\rrbracket$  from $D\llbracket B/X\rrbracket,\Delta$ (observe that the role of $A$ and $B$ in assumptions and conclusions is inverted), which we likewise indicate as $\SF{X}{D}{\mathscr{D}}$.

We will use the {\em same} notation also for the substitution of $A$ for $X$ in $C$ writing $\SF{X}{C}{A}$ for $C\llbracket A/X\rrbracket$. Beside saving brackets, this uniform notation highlights the functorial nature of pn-$X$ formulas, that is the fact that to each p-$X$ (resp.~n-$X$) formula $C$ one can associate a functor $\SF{X}{C}{\cdot}$ that applied to a formula $A$ yields the formula $C\llbracket A/X\rrbracket \equiv \SF{X}{C}{A}$ as value and applied to a derivation $\mathscr{D}$ of $B$ from $A,\Delta$ yields a derivation $\SF{X}{C}{\mathscr{D}}$ of $C\llbracket B/X\rrbracket$  from $C\llbracket A/X\rrbracket,\Delta$ (resp.~of $C\llbracket A/X\rrbracket$ from  $C\llbracket B/X\rrbracket, \Delta$) as value. 

Leaving the assumptions in $\Delta$ implicit, we thus have that, for $C$ p-$X$ and $D$ n-$X$, 
\begingroup\makeatletter\def\f@size{10}\check@mathfonts
$${\small\SF{X}{C}{\mathscr{D}} \equiv {\small \def\extraVskip{0pt}\AXC{$C\llbracket A/X\rrbracket$}\noLine\UIC{$\SF{X}{C}{\makebox[2em][c]{$\mathscr{D}$}}$}\noLine\UIC{$C\llbracket B/X\rrbracket$}\DP} \equiv {\small \def\extraVskip{0pt}\AXC{$\SF{X}{C}{A}$}\noLine\UIC{$\SF{X}{C}{\mathscr{D}}$}\noLine\UIC{$\SF{X}{C}{B}$}\DP}
\qquad
\SF{X}{D}{\mathscr{D}} \equiv {\small \def\extraVskip{0pt}\AXC{$D\llbracket B/X\rrbracket$}\noLine\UIC{$\SF{X}{D}{\makebox[2em][c]{$\mathscr{D}$}}$}\noLine\UIC{$D\llbracket A/X\rrbracket$}\DP} \equiv {\small \def\extraVskip{0pt}\AXC{$\SF{X}{D}{B}$}\noLine\UIC{$\SF{X}{D}{\mathscr{D}}$}\noLine\UIC{$\SF{X}{D}{A}$}\DP}
}$$
\endgroup
which we sometimes shorten further to 
\begingroup\makeatletter\def\f@size{10}\check@mathfonts
$$
\cexpD{X}{C}{\AXC{$A$}\noLine\UIC{$\mathscr{D}$}\noLine\UIC{$B$}\DP}
\qquad\qquad
\cexpD{X}{D}{\AXC{$A$}\noLine\UIC{$\mathscr{D}$}\noLine\UIC{$B$}\DP}
$$
\endgroup
(In the case of the n-$X$ formula $D$, we warn again the reader that the assumptions of the $D$-expansion of $\D D$ are $D\llbracket B/X\rrbracket, \Delta$ and not, as the notation may suggest, $D\llbracket A/X\rrbracket, \Delta$, and similarly for the conclusion.)

If $\Gamma$ is the multiset of formulas $C_{1},\dots, C_{n}$, then by $\Gamma\ldbrack A/X\rdbrack$ we indicate the multiset of formulas $C_1\ldbrack A/X\rdbrack,\ldots, C_n\ldbrack A/X\rdbrack$. Likewise, we use $\SF{X}{\Gamma}{A}$ for $\SF{X}{C_1}{A},\ldots, \SF{X}{C_n}{A}$ and $\SF{X}{\Gamma}{\mathscr{D}}$ for $\SF{X}{C_1}{\mathscr{D}},\ldots,\SF{X}{C_n}{\mathscr{D}}$. 

\begin{remark} The following basic facts about substitution:
\begin{enumerate}
 \item If $X$ does not occur in $C$, then  $C\llbracket A/X\rrbracket \equiv C$ for all $A$;
\item For all $A$, $X$ and $C$,  $(\forall X C) \llbracket A/X\rrbracket \equiv \forall X C$;
\item If $X$ does not occur in $C$, and $Y$ does not occur in $A$, then  for all $F$ 
\begingroup\makeatletter\def\f@size{10}\check@mathfonts
$$F\llbracket C /Y\rrbracket \llbracket A /X\rrbracket \equiv F \llbracket A /X\rrbracket \llbracket C /Y\rrbracket$$
\endgroup
\end{enumerate}
will be therefore expressed as follows:
\begin{enumerate}
 \item If $X$ does not occur in $C$, then for all $A$ $\SF{X}{C}{A} \equiv C$;
 \item For all $A$, $X$ and $C$, $\SF{X}{\forall X C}{A} \equiv \forall X C$;
\item If $X$ does not occur in $C$, nor $Y$ in $A$, then for all $F$  
\begingroup\makeatletter\def\f@size{10}\check@mathfonts
$$\SF{X}{F\llbracket C /Y\rrbracket }{A} \equiv \SF{X}{F}{A}\llbracket C /Y\rrbracket$$
\endgroup
\end{enumerate}
\end{remark}

\begin{definition}[$C$-expansion of $\mathscr{D}$ relative to $X$]\label{cexp}
If $C$ is p-$X$ (resp.~n-$X$) and $\mathscr{D}$ is a derivation of $B$ from undischarged assumptions $A,\Delta$, we call the {\em $C$-expansion of $\mathscr{D}$ relative to $X$}, notation $\SF{X}{C}{\D D}$, the derivation  of $C\ldbrack B/X\rdbrack$ from $C\ldbrack A/X\rdbrack, \Delta$ (resp.~of $C\ldbrack A/X\rdbrack$ from $C\ldbrack B/X\rdbrack, \Delta$) defined as follows: 
\begin{itemize}
 \item[A.] If $X$ does not occur in $C$, then $\SF{X}{C}{\mathscr{D}}$ just consists of the assumption of $C$.
\item[B.] Otherwise $\SF{X}{C}{\mathscr{D}}$ is defined by induction on $C$:
\begin{enumerate}
\item If $C\equiv X$ then $\SF{X}{C}{\mathscr{D}} \equiv \D D$; 
\item If $C \equiv \forall Y F$ then, as $C$ is p-$X$ (resp.~n-$X$), $F$ is p-$X$ (resp.~n-$X$) too, we define 
\begingroup\makeatletter\def\f@size{10}\check@mathfonts
\begin{lrbox}{\mypti}% Store prooftree in \mypti
\begin{varwidth}{\linewidth}
$\cexpD{X}{F}{\AXC{$A$}\noLine\UIC{$\mathscr{D}$}\noLine\UIC{$B$}\DP}$
\end{varwidth}
\end{lrbox}
$$
\SF{X}{\forall Y F}{\mathscr{D}} \equiv 
\AXC{$\forall Y\SF{X}{F}{A}$}
\RightLabel{\footnotesize$\forall$E}
\UIC{\usebox{\mypti}}
\RightLabel{\footnotesize$\forall$I}
\UIC{$\forall Y\SF{X}{F}{B}$}\DisplayProof 
\equiv
\AXC{$\SF{X}{\forall YF}{A}$}
\RightLabel{\footnotesize$\forall E$}
\UIC{\usebox{\mypti}}
\RightLabel{\footnotesize$\forall$I}
\UIC{$\SF{X}{\forall Y F}{B}$}
\DisplayProof
$$
$$\left ( \ \text{resp.~  }
\SF{X}{\forall Y F}{\mathscr{D}} \equiv 
\AXC{$\forall Y\SF{X}{F}{B}$}
\RightLabel{\footnotesize$\forall$E}
\UIC{\usebox{\mypti}}
\RightLabel{\footnotesize$\forall$I}
\UIC{$\forall Y\SF{X}{F}{A}$}\DisplayProof 
\equiv
\AXC{$\SF{X}{\forall YF}{B}$}
\RightLabel{\footnotesize$\forall E$}
\UIC{\usebox{\mypti}}
\RightLabel{\footnotesize$\forall$I}
\UIC{$\SF{X}{\forall Y F}{A}$}
\DisplayProof \ \right )
$$\endgroup
(where the substitution $(\forall Y F)[A/X]\equiv \forall Y' (F'[A/X])$ might require a renaming of the bound variables of $F$, cf.~remark~\ref{alfa}i.{} above).
\item If $C\equiv F\supset G$ then, $F$ is n-$X$ (resp.~p-$X$) and $G$ is p-$X$ (resp.~n-$X$), so we can define
\begingroup\makeatletter\def\f@size{10}\check@mathfonts
\begin{lrbox}{\mypti}% Store prooftree in \mypti
\begin{varwidth}{\linewidth}
$\cexpD{X}{G}{\AXC{$A$}\noLine\UIC{$\mathscr{D}$}\noLine\UIC{$B$}\DP}$
\end{varwidth}
\end{lrbox}
\begin{lrbox}{\mypto}% Store prooftree in \mypto
\begin{varwidth}{\linewidth}
$\cexpD{X}{F}{\AXC{$A$}\noLine\UIC{$\mathscr{D}$}\noLine\UIC{$B$}\DP}$
\end{varwidth}
\end{lrbox}
$$\SF{X}{C}{\mathscr{D}} \equiv \def\defaultHypSeparation{\hskip .45em}\AxiomC{$\begin{matrix} \ \\ \ \\  \SF{X}{F}{A}\supset\SF{X}{G}{A}\end{matrix}$}
\AxiomC{$\stackrel{n}{\usebox{\mypto}}$}
\RightLabel{\footnotesize$\supset$E}
\BinaryInfC{\usebox{\mypti}}
\RightLabel{\footnotesize$\supset$I $(n)$}
\UIC{$\SF{X}{G}{B}\supset \SF{X}{G}{B}$}
\DisplayProof\equiv
\AxiomC{$\begin{matrix} \ \\ \ \\  \SF{X}{F\supset G}{A}\end{matrix}$}
\AxiomC{$\stackrel{n}{\usebox{\mypto}}$}
\RightLabel{\footnotesize$\supset$E}
\BinaryInfC{\usebox{\mypti}}
\RightLabel{\footnotesize$\supset$I $(n)$}
\UIC{$\SF{X}{F\supset G}{B}$}
\DisplayProof$$
$$\left (\text{resp.~} \SF{X}{C}{\mathscr{D}} \equiv \def\defaultHypSeparation{\hskip .45em}\AxiomC{$\begin{matrix} \ \\ \ \\  F\supset\SF{X}{G}{B}\end{matrix}$}
\AxiomC{$\stackrel{n}{\usebox{\mypto}}$}
\RightLabel{\footnotesize$\supset$E}
\BinaryInfC{\usebox{\mypti}}
\RightLabel{\footnotesize$\supset$I $(n)$}
\UIC{$F\supset \SF{X}{G}{A}$}
\DisplayProof\equiv
\AxiomC{$\begin{matrix} \ \\ \ \\  \SF{X}{F\supset G}{B}\end{matrix}$}
\AxiomC{$\stackrel{n}{\usebox{\mypto}}$}
\RightLabel{\footnotesize$\supset$E}
\BinaryInfC{\usebox{\mypti}}
\RightLabel{\footnotesize$\supset$I $(n)$}
\UIC{$\SF{X}{F\supset G}{A}$}
\DisplayProof\right )$$
\endgroup	
\end{enumerate}
\end{itemize}
\end{definition}

\medskip
In the proof of the main theorem in section \ref{nat-ss} we will need the following:
\begin{lemma}\label{lemmasafe}
If $\mathscr{D}$ is a derivation of $B$ from $A,\Delta$ and $X$ does not occur in $F$, then for all $Y$ not occurring free in $\D D$, the result of substituting $F$ for $Y$ in the $C$-expansion of $\mathscr{D}$ (relative to $X$) is equal to the $C\ldbrack F/Y\rdbrack$-expansion of $\mathscr{D}$ (relative to $X$):  
\begingroup\makeatletter\def\f@size{10}\check@mathfonts
$$\SF{X}{C\ldbrack F/Y\rdbrack}{\mathscr{D}}\equiv \SF{X}{C}{\mathscr{D}}\ldbrack F/Y\rdbrack$$
\endgroup
\end{lemma}
\begin{proof}
By induction on $C$, by observing that, as $X$ does not occur in $F$, the derivation $\SF{X}{F}{\mathscr{D}}$  consists solely of the assumption of $F$.\end{proof}

\subsection{The naturality condition}\label{nat-ss}

The core insight from category theory is the following: the system \Nd can be seen as a syntactic category, whose objects are $\C L^{2}$ formulas and whose morphisms are \Nd derivations. Then, given a variable $X$ we can associate to any p-$X$ (resp.~n-$X$) formula $C$ a functor, whose application to a formula $A$ gives $\SF{X}{C}{A}$ as value, and whose application to a derivation $\mathscr{D}$ gives $\SF{X}{C}{\mathscr{D}}$ as value.  Moreover, given a p-$X$ (resp.~n-$X$), $X$-safe derivation $\mathscr{D}$ of $C$ from $\Gamma$, the operation which associates to any formula $A$ the derivation $\mathscr{D}\ldbrack A/X\rdbrack$ (that we will abbreviate with $\mathscr{D}\ldbrack A\rdbrack$) yields a family of morphisms $\theta_A$ from $\SF{X}{\Gamma}{A}$ to $\SF{X}{C}{A}$. In the terminology of category theory, such a family of morphisms is a natural transformation between the functors associated to $\Gamma$ and $C$, provided that for any derivation $\mathscr{D}'$ of $B$ from $A,\Delta$ the following diagram commutes:
\begingroup\makeatletter\def\f@size{10}\check@mathfonts
\begin{equation*}\label{nat}
\begin{matrix}
\vcenter{\xymatrix{
\SF{X}{\Gamma}{A} \ar[r]^{\SD{X}{\mathscr{D}}{A}} \ar[d]_{\SF{X}{\Gamma}{\mathscr{D}'}} & \SF{X}{C}{A} \ar[d]^{\SF{X}{C}{\mathscr{D}'}} \\
\SF{X}{\Gamma}{B} \ar[r]_{\SD{X}{\mathscr{D}}{B}} & \SF{X}{C}{B}
}}
& \ & \ & \ &  \left ( \text{ resp.~ }\vcenter{
\xymatrix{
\SF{X}{\Gamma}{B} \ar[r]^{\SD{X}{\mathscr{D}}{B}} \ar[d]_{\SF{X}{\Gamma}{\mathscr{D}'}} & \SF{X}{C}{B} \ar[d]^{\SF{X}{C}{\mathscr{D}'}} \\
\SF{X}{\Gamma}{A} \ar[r]_{\SD{X}{\mathscr{D}}{A}} & \SF{X}{C}{A}
}} \ \ \right )
\end{matrix}
\end{equation*} 
\endgroup

Using the notions so far introduced, this can be expressed as follows:
\begin{definition}[naturality condition]\label{natcon}
Let $\mathscr{D}$ be a p-$X$ (resp.~n-$X$) derivation of $C$ from $\Gamma$. We say that $\mathscr{D}$ is {\em natural in $X$} iff for any derivation $\mathscr{D}'$ of $B$ from $A,\Delta$, the composition of $\SD{X}{\mathscr{D}}{A}$ with $\SF{X}{C}{\mathscr{D}'}$ is $\beta$-equal to the composition of $\SF{X}{\Gamma}{\mathscr{D}'}$ with $\SD{X}{\mathscr{D}}{B}$, that is iff the following holds: 
\begingroup\makeatletter\def\f@size{10}\check@mathfonts
\begin{lrbox}{\mypti}% Store prooftree in \mypti
\begin{varwidth}{\linewidth}
$\cexpD{X}{C}{\AXC{$A$}\noLine\UIC{$\mathscr{D}'$}\noLine\UIC{$B$}\DP}$
\end{varwidth}
\end{lrbox}
\begin{lrbox}{\mypto}% Store prooftree in \mypti
\begin{varwidth}{\linewidth}
$\cexpD{X}{\Gamma}{\AXC{$A$}\noLine\UIC{$\mathscr{D}'$}\noLine\UIC{$B$}\DP}$
\end{varwidth}
\end{lrbox}
$$\def\extraVskip{0pt}\AXC{$\phantom{A}$}\noLine\UIC{$\phantom{A}$}\noLine\UIC{$\SFI{X}{\Gamma}{A}$}\noLine\UIC{$\SDI{X}{\mathscr{D}}{A}$}\noLine\UIC{$\SFI{X}{C}{A}$}\noLine\UIC{$\SFI{X}{C}{\mathscr{D}'}$}\noLine\UIC{$\SFI{X}{C}{B}$}\DP\equiv\def\extraVskip{0pt}\AXC{$\SFI{X}{\Gamma}{A}$}\noLine\UIC{$\SDI{X}{\mathscr{D}}{A}$}\noLine\UIC{\usebox{\mypti}}\DP\qquad  \text{{\Large$=_\beta^{2\supset}$}}\qquad
\def\extraVskip{0pt}\AXC{\usebox{\mypto}}\noLine\UIC{$\SDI{X}{\mathscr{D}}{B}$}\noLine\UIC{$\SFI{X}{C}{B}$}\DP\equiv\def\extraVskip{0pt}\AXC{$\SFI{X}{\Gamma}{A}$}\noLine\UIC{$\SFI{X}{\Gamma}{\mathscr{D}'}$}\noLine\UIC{$\SFI{X}{\Gamma}{B}$}\noLine\UIC{$\SDI{X}{\mathscr{D}}{B}$}\noLine\UIC{$\SFI{X}{C}{B}$}\DP$$
\begin{lrbox}{\mypti}% Store prooftree in \mypti
\begin{varwidth}{\linewidth}
$\cexpD{X}{C}{\AXC{$A$}\noLine\UIC{$\mathscr{D}'$}\noLine\UIC{$B$}\DP}$
\end{varwidth}
\end{lrbox}
\begin{lrbox}{\mypto}% Store prooftree in \mypti
\begin{varwidth}{\linewidth}
$\cexpD{X}{\Gamma}{\AXC{$A$}\noLine\UIC{$\mathscr{D}'$}\noLine\UIC{$B$}\DP}$
\end{varwidth}
\end{lrbox}
$$ \left ( \ \text{resp.~  }\def\extraVskip{0pt}\AXC{$\phantom{A}$}\noLine\UIC{$\phantom{A}$}\noLine\UIC{$\SFI{X}{\Gamma}{B}$}\noLine\UIC{$\SDI{X}{\mathscr{D}}{B}$}\noLine\UIC{$\SFI{X}{C}{B}$}\noLine\UIC{$\SFI{X}{C}{\mathscr{D}'}$}\noLine\UIC{$\SFI{X}{C}{A}$}\DP\equiv\def\extraVskip{0pt}\AXC{$\SFI{X}{\Gamma}{B}$}\noLine\UIC{$\SDI{X}{\mathscr{D}}{B}$}\noLine\UIC{\usebox{\mypti}}\DP\qquad  \text{{\Large$=_\beta^{2\supset}$}}\qquad
\def\extraVskip{0pt}
\AXC{\usebox{\mypto}}
\noLine
\UIC{$\SDI{X}{\mathscr{D}}{A}$}
\noLine
\UIC{$\SFI{X}{C}{A}$}
\DP
\equiv\def\extraVskip{0pt}
\AXC{$\SFI{X}{\Gamma}{B}$}
\noLine\UIC{$\SFI{X}{\Gamma}{\mathscr{D}'}$}\noLine\UIC{$\SFI{X}{\Gamma}{A}$}\noLine\UIC{$\SDI{X}{\mathscr{D}}{A}$}\noLine\UIC{$\SFI{X}{C}{A}$}\DP \ \ \right )$$
\endgroup
\end{definition}

\medskip
\begin{theorem}\label{naturality-th}
If $\mathscr{D}$ is an $X$-safe, p-$X$ (resp.~n-$X$) derivation, then $\mathscr{D}$ is natural in $X$.
\end{theorem}
\begin{proof}
We will actually prove a slightly stronger claim: if $\D D$ is a pn-$X$ derivation, then the following equation holds (where $\Gamma$ (resp.~$\Delta$) is the set of p-$X$ (resp.~n-$X$) hypotheses of $\D D$):
\begingroup\makeatletter\def\f@size{10}\check@mathfonts
\begin{lrbox}{\mypti}% Store prooftree in \mypti
\begin{varwidth}{\linewidth}
$\cexpD{X}{C}{\AXC{$A$}\noLine\UIC{$\mathscr{D}'$}\noLine\UIC{$B$}\DP}$
\end{varwidth}
\end{lrbox}
\begin{lrbox}{\mypto}% Store prooftree in \mypti
\begin{varwidth}{\linewidth}
$\cexpD{X}{\Gamma}{\AXC{$A$}\noLine\UIC{$\mathscr{D}'$}\noLine\UIC{$B$}\DP}$
\end{varwidth}
\end{lrbox}
\begin{lrbox}{\myptu}% Store prooftree in \mypti
\begin{varwidth}{\linewidth}
$\cexpD{X}{\Delta}{\AXC{$A$}\noLine\UIC{$\mathscr{D}'$}\noLine\UIC{$B$}\DP}$
\end{varwidth}
\end{lrbox}
$$\def\extraVskip{0pt}
\AXC{$\phantom{A}$}
\noLine
\UIC{$\phantom{A}$}
\noLine
\UIC{$ \SF{X}{\Gamma}{A}$}
\AXC{$\SFI{X}{\Delta}{B}$}
\noLine
\UIC{$\SFI{X}{\Delta}{\D D'}$}
\noLine
\UIC{$\SFI{X}{\Delta}{A}$}
\noLine
\BIC{$\SDI{X}{\mathscr{D}}{A}$}
\noLine
\UIC{$\SFI{X}{C}{A}$}
\noLine
\UIC{$\SFI{X}{C}{\mathscr{D}'}$}
\noLine
\UIC{$\SFI{X}{C}{B}$}
\DP
\ \ =^{2\supset}_{\beta} \ \
\AXC{$\SFI{X}{\Gamma}{A}$}
\noLine
\UIC{$\SFI{X}{\Gamma}{\D D'}$}
\noLine
\UIC{$\SFI{X}{\Gamma}{B}$}
\AXC{$\phantom{A}$}
\noLine
\UIC{$\phantom{A}$}
\noLine
\UIC{$\SF{X}{\Delta}{B}$}
\noLine
\BIC{$\SDO{X}{\mathscr{D}}{B}$}
\noLine
\UIC{$\SFO{X}{C}{B}$}
\DP$$
$$
\left ( \text{ resp.~}\def\extraVskip{0pt}
\AXC{$\SFI{X}{\Gamma}{A}$}
\noLine
\UIC{$\SFI{X}{\Gamma}{\D D'}$}
\noLine
\UIC{$\SFI{X}{\Gamma}{B}$}
\AXC{$\phantom{A}$}
\noLine
\UIC{$\phantom{A}$}
\noLine
\UIC{$ \SF{X}{\Delta}{B}$}
\noLine
\BIC{$\SDI{X}{\mathscr{D}}{B}$}
\noLine
\UIC{$\SFI{X}{C}{B}$}
\noLine
\UIC{$\SFI{X}{C}{\mathscr{D}'}$}
\noLine
\UIC{$\SFI{X}{C}{A}$}
\DP
\ \ =^{2\supset}_{\beta} \ \
\AXC{$\phantom{B}$}
\noLine
\UIC{$\phantom{B}$}
\noLine
\UIC{$\SF{X}{\Gamma}{A}$}
\AXC{$\SFI{X}{\Delta}{B}$}
\noLine
\UIC{$\SFI{X}{\Delta}{\D D'}$}
\noLine
\UIC{$\SFI{X}{\Delta}{A}$}
\noLine
\BIC{$\SDI{X}{\mathscr{D}}{A}$}
\noLine
\UIC{$\SFI{X}{C}{A}$}
\DP
\right )
$$
\endgroup

Without loss of generality we can assume $\D D$ normal; we argue by induction on the number of inference rules applied in $\mathscr{D}$. We limit ourselves to some cases:
\begin{itemize}
\item if $\D D$ consists of the assumption of a p-$X$ (resp.~n-$X$) formula $C$ then, for all derivations $\mathscr{D}'$ of $B$ from $A$, we obviously have 
\begingroup\makeatletter\def\f@size{10}\check@mathfonts
\begin{lrbox}{\mypti}% Store prooftree in \mypti
\begin{varwidth}{\linewidth}
$\cexpD{X}{C}{\AXC{$A$}\noLine\UIC{$\mathscr{D}'$}\noLine\UIC{$B$}\DP}$
\end{varwidth}
\end{lrbox}
$$ \usebox{\mypti}=_{\beta}^{2\supset} \usebox{\mypti} $$
\endgroup
\item if $\D D\equiv {\small
\AXC{$\Gamma$}\noLine\UIC{$\D D_1$}\noLine\UIC{$G$}\RightLabel{\footnotesize$\forall I$}\UIC{$\forall Y G$}\DisplayProof}$ and $G$ is p-$X$ then, for all derivations $\mathscr{D}'$ of $B$ from $A$, supposing by remark~\ref{alfa}ii.{} that no bound variable of $\D D$ occurs free in $\D D'$, we have that
\begingroup\makeatletter\def\f@size{10}\check@mathfonts
\begin{lrbox}{\mypti}% Store prooftree in \mypti
\begin{varwidth}{\linewidth}
$\cexpD{X}{G}{\AXC{$A$}\noLine\UIC{$\mathscr{D}'$}\noLine\UIC{$B$}\DP}$
\end{varwidth}
\end{lrbox}
\begin{lrbox}{\mypta}% Store prooftree in \mypti
\begin{varwidth}{\linewidth}
$\cexpD{X}{\Gamma}{\AXC{$A$}\noLine\UIC{$\mathscr{D}'$}\noLine\UIC{$B$}\DP}$
\end{varwidth}
\end{lrbox}
$$\def\extraVskip{0pt}\def\defaultHypSeparation{\hskip 1em}
\AXC{$\SF{X}{\Gamma}{A}$}
\AXC{$\SFI{X}{\Delta}{B}$}
\noLine
\UIC{$\SFI{X}{\Delta}{\D D'}$}
\noLine
\UIC{$\SFI{X}{\Delta}{A}$}
\noLine
\BIC{$\SDI{X}{\D D_1}{A}$}\noLine\UIC{$\SFI{X}{G}{A}$}\def\extraVskip{2pt}
\RightLabel{\footnotesize$\forall I$}\UIC{$\SFI{X}{\forall Y G}{A}$}\def\extraVskip{0pt}
\noLine\UIC{$\SFI{X}{\forall Y G}{\D D'}$}\noLine\UIC{$\SFI{X}{\forall YG}{ B} $}\DisplayProof
\ \ \equiv\ \
\def\extraVskip{0pt}\AXC{$\SF{X}{\Gamma}{A}$}\AXC{$\SFI{X}{\Delta}{B}$}
\noLine
\UIC{$\SFI{X}{\Delta}{\D D'}$}
\noLine
\UIC{$\SFI{X}{\Delta}{A}$}
\noLine\BIC{$\SDI{X}{\D D_1}{A}$}\noLine\UIC{$\SFI{X}{G}{A}$}\def\extraVskip{2pt}\RightLabel{\footnotesize$\forall I$}\UIC{$\forall Y\SF{X}{ G}{A}$}\RightLabel{\footnotesize$\forall E$}\UIC{\usebox{\mypti}}\RightLabel{\footnotesize$\forall I$}\UIC{$\forall Y\SF{X}{G}{ B }$}\DisplayProof
\ \ =_{\beta}^{2}\ \ 
\def\extraVskip{0pt}\AXC{$\SF{X}{\Gamma}{A}$}\AXC{$\SFI{X}{\Delta}{B}$}
\noLine
\UIC{$\SFI{X}{\Delta}{\D D'}$}
\noLine
\UIC{$\SFI{X}{\Delta}{A}$}
\noLine\BIC{$\SDI{X}{\D D_1}{A}$}\noLine\UIC{\usebox{\mypti}}\def\extraVskip{2pt}\RightLabel{\footnotesize$\forall I$}\UIC{$\SF{X}{\forall YG}{B}$}\DisplayProof
\quad\stackrel{\text{\fbox{i.h.$\!$}}}{=_{\beta}^{2\supset}}\quad$$
$$\ \ \stackrel{\text{\fbox{i.h.$\!$}}}{=_{\beta}^{2\supset}}\ \ \def\defaultHypSeparation{\hskip 1em}\def\extraVskip{0pt}\AXC{$\SFI{X}{\Gamma}{A}$}\noLine\UIC{$\SFI{X}{\Gamma}{\mathscr{D}'}$}\noLine\UIC{$\SFI{X}{\Gamma}{B}$}\AXC{$\SF{X}{\Delta}{B}$}\noLine\BIC{$\SDI{X}{\D D_1}{B}$}\noLine\UIC{$\SFI{X}{G}{B}$}\def\extraVskip{2pt}\RightLabel{\footnotesize$\forall I$}\UIC{$\SF{X}{\forall YG}{B}$}\DisplayProof
\ \ \equiv\ \  \def\extraVskip{0pt}
\AXC{$\SFI{X}{\Gamma}{A}$}\noLine\UIC{$\SFI{X}{\Gamma}{\mathscr{D}'}$}\noLine\UIC{$\SFI{X}{\Gamma}{B}$}\AXC{$\SF{X}{\Delta}{B}$}\def\extraVskip{0pt}\noLine\BIC{$\SD{X}{\D D_1}{B}$}\noLine\UIC{$\SF{X}{G}{B}$}\def\extraVskip{2pt}\RightLabel{\footnotesize$\forall I$}\UIC{$\SF{X}{\forall YG}{B}$}\DisplayProof
$$
\endgroup
\item if $\D D\equiv {\small\AXC{$\Gamma   $}\noLine\UIC{$\D D_1 $}\noLine\UIC{$\forall Y G$}\RightLabel{\footnotesize$\forall E$}\UIC{$G\llbracket F/Y\rrbracket$}\DisplayProof}$, with $G$ p-$X$, then $X$ does not occur in $F$ (since $\D D$ is $X$-safe). By remark \ref{alfa}ii., we can suppose that $Y$ does not occur free in $\D D'$ and thus, by lemma \ref{lemmasafe}, for all $\D D'$ of $B$ from $A$, $\SF{X}{G\ldbrack F/Y\rdbrack}{\mathscr{D}'}\equiv \SF{X}{G}{\mathscr{D}'}\ldbrack F/Y\rdbrack$. Hence
\begingroup\makeatletter\def\f@size{10}\check@mathfonts
\begin{lrbox}{\mypti}% Store prooftree in \mypti
\begin{varwidth}{\linewidth}
$\cexpD{X}{\forall Y G}{\AXC{$A$}\noLine\UIC{$\mathscr{D}'$}\noLine\UIC{$B$}\DP}$
\end{varwidth}
\end{lrbox}
\begin{lrbox}{\mypto}% Store prooftree in \mypti
\begin{varwidth}{\linewidth}
$\cexpD{X}{G}{\AXC{$A$}\noLine\UIC{$\mathscr{D}'$}\noLine\UIC{$B$}\DP}$
\end{varwidth}
\end{lrbox}
\begin{lrbox}{\mypta}% Store prooftree in \mypti
\begin{varwidth}{\linewidth}
$\cexpD{X}{\Gamma}{\AXC{$A$}\noLine\UIC{$\mathscr{D}'$}\noLine\UIC{$B$}\DP}$
\end{varwidth}
\end{lrbox}
\begin{lrbox}{\mypte}% Store prooftree in \mypti
\begin{varwidth}{\linewidth}
$\cexpD{X}{G\llbracket F/Y\rrbracket}{\AXC{$A$}\noLine\UIC{$\mathscr{D}'$}\noLine\UIC{$B$}\DP}$
\end{varwidth}
\end{lrbox}
$$\def\defaultHypSeparation{\hskip 1em}\def\extraVskip{0pt}\AXC{$\SF{X}{\Gamma}{A}$}\AXC{$\SFI{X}{\Delta}{B}$}\noLine\UIC{$\SFI{X}{\Delta}{\mathscr{D}'}$}\noLine\UIC{$\SFI{X}{\Delta}{A}$}\noLine\BIC{$\SDI{X}{\D D_{1}}{A}$}\noLine\UIC{$\SF{X}{\forall YG}{A}$}\def\extraVskip{2pt}\RightLabel{\footnotesize$\forall E$}\UIC{$\SFI{X}{G\llbracket F/Y\rrbracket}{A}$}\def\extraVskip{0pt}\noLine\UIC{$\SFI{X}{G\llbracket F/Y\rrbracket}{\D D'}$}\noLine\UIC{$\SFI{X}{G\llbracket F/Y\rrbracket}{B}$}\DisplayProof 
\ \ \equiv\ \
\def\extraVskip{0pt}\AXC{$\SF{X}{\Gamma}{A}$}\AXC{$\SFI{X}{\Delta}{B}$}\noLine\UIC{$\SFI{X}{\Delta}{\mathscr{D}'}$}\noLine\UIC{$\SFI{X}{\Delta}{A}$}\noLine
\BIC{$\SD{X}{\D D_{1}}{A}$}\noLine\UIC{$\forall Y\SF{X}{G}{A}$}\def\extraVskip{2pt}\RightLabel{\footnotesize$\forall E$}\UIC{\usebox{\mypte}}\DisplayProof 
\ \ \equiv\ \ 
\def\extraVskip{0pt}\AXC{$\SF{X}{\Gamma}{A}$}
\AXC{$\SFI{X}{\Delta}{B}$}\noLine\UIC{$\SFI{X}{\Delta}{\mathscr{D}'}$}\noLine\UIC{$\SFI{X}{\Delta}{A}$}\noLine\BIC{$\SD{X}{\D D_{1}}{A}
$}\noLine\UIC{$\forall Y\SF{X}{G}{A}$}\def\extraVskip{2pt}\RightLabel{\footnotesize$\forall E$}\UIC{\usebox{\mypto}$\llbracket F/Y\rrbracket$}\DisplayProof\ \ =_{\beta}^{2}\ \  $$
$$
=_{\beta}^{2}\ 
\def\defaultHypSeparation{\hskip 1em}\def\extraVskip{0pt}\AXC{$\SF{X}{\Gamma}{A}$}\AXC{$\SFI{X}{\Delta}{B}$}\noLine\UIC{$\SFI{X}{\Delta}{\mathscr{D}'}$}\noLine\UIC{$\SFI{X}{\Delta}{A}$}\noLine\BIC{$\SD{X}{\D D_{1}}{A}$}\noLine\UIC{$\forall Y\SF{X}{G}{A}$}\def\extraVskip{2pt}\RightLabel{\footnotesize$\forall E$}\UIC{\usebox{\mypto}}\RightLabel{\footnotesize$\forall I$}\UIC{$\forall Y\SF{X}{G}{B}$}\RightLabel{\footnotesize$\forall E$}\UIC{$\SF{X}{G}{B}\llbracket F/Y\rrbracket$}\DisplayProof
\ \ \equiv\ \
 \def\extraVskip{0pt}\AXC{$\SF{X}{\Gamma}{A}$}\AXC{$\SFI{X}{\Delta}{B}$}\noLine\UIC{$\SFI{X}{\Delta}{\mathscr{D}'}$}\noLine\UIC{$\SFI{X}{\Delta}{A}$}\noLine\BIC{$\SD{X}{\D D_{1}}{A}$}\noLine\UIC{\usebox{\mypti}}\def\extraVskip{2pt}\RightLabel{\footnotesize$\forall E$}\UIC{$\SF{X}{G\llbracket F/Y\rrbracket}{B}$}\DisplayProof
\ \ \stackrel{\text{\fbox{i.h.$\!$}}}{=_{\beta}^{2\supset}}\ \
 \def\extraVskip{0pt}\AXC{$\SFI{X}{\Gamma}{A}$}\noLine\UIC{$\SFI{X}{\Gamma}{\mathscr{D}'}$}\noLine\UIC{$\SFI{X}{\Gamma}{B}$}\AXC{$\SF{X}{\Delta}{B}$}\noLine\BIC{$\SD{X}{\D D_{1}}{B}$}\noLine\UIC{$\SF{X}{\forall YG}{B}$}\def\extraVskip{2pt}\RightLabel{\footnotesize$\forall E$}\UIC{$\SF{X}{G\llbracket F/Y\rrbracket}{B}$}\DisplayProof$$
\endgroup
\item if $\D D\equiv {\small\AXC{$\Gamma\ [\stackrel{n}{F}]$}\noLine\UIC{$\D D_1$}\noLine\UIC{$G$}\RightLabel{\footnotesize$\supset$I $(n)$}\UIC{$F\supset G$}\DisplayProof}$, where $F\supset G$ is p-$X$, then $F$ is n-$X$ and $G$ is p-$X$. Hence we have that
\begingroup\makeatletter\def\f@size{10}\check@mathfonts
\begin{lrbox}{\mypto}% Store prooftree in \mypti
\begin{varwidth}{\linewidth}
$\cexpD{X}{F\supset G}{\AXC{$A$}\noLine\UIC{$\mathscr{D}'$}\noLine\UIC{$B$}\DP}$
\end{varwidth}
\end{lrbox}
\begin{lrbox}{\mypta}% Store prooftree in \mypti
\begin{varwidth}{\linewidth}
$\cexpD{X}{F}{\AXC{$A$}\noLine\UIC{$\mathscr{D}'$}\noLine\UIC{$B$}\DP}$
\end{varwidth}
\end{lrbox}
\begin{lrbox}{\mypti}% Store prooftree in \mypti
\begin{varwidth}{\linewidth}
$\cexpD{X}{G}{\AXC{$A$}\noLine\UIC{$\mathscr{D}'$}\noLine\UIC{$B$}\DP}$
\end{varwidth}
\end{lrbox}
$$\def\extraVskip{0pt}\def\defaultHypSeparation{\hskip .0em}
\AXC{$\SF{X}{\Gamma}{A}$\!\!\!\!\!}
\AXC{$\ \ \, [ \stackrel{n}{\SFI{X}{F}{A}} ] $}
\AXC{$\SFI{X}{\Delta}{B}$}
\noLine
\UIC{$\SFI{X}{\Delta}{\D D'}$}
\noLine
\UIC{$\SFI{X}{\Delta}{A}$}
\noLine\TrinaryInfC{$\SDI{X}{\D D_1 }{A}$}\noLine\UIC{$\SFI{X}{G}{A}$}\def\extraVskip{2pt}\RightLabel{\footnotesize$\supset$I $(n)$}
\UIC{$\SFI{X}{F\supset G}{A}$}\def\extraVskip{0pt}\noLine\UIC{$\SFI{X}{F\supset G}{\D D'}$}\noLine\UIC{$\SFI{X}{F\supset G}{B}$}\DisplayProof 
\ \ \equiv \ \ 
\def\extraVskip{0pt}
\AXC{$\SF{X}{\Gamma}{A} \!\!\!\!\!$}
\AXC{$\ \ \, [\stackrel{n}{\SFI{X}{F}{A}}]$}
\AXC{$\SFI{X}{\Delta}{B}$}
\noLine
\UIC{$\SFI{X}{\Delta}{\D D'}$}
\noLine
\UIC{$\SFI{X}{\Delta}{A}$}
\noLine\TrinaryInfC{$\SDI{X}{\D D_1}{ A }$}\noLine\UIC{$\SFI{X}{G}{A}$}\def\extraVskip{2pt}\RightLabel{\footnotesize$\supset$I $(n)$}\UIC{$\SF{X}{F\supset G}{A}$}
\def\extraVskip{0pt}\AXC{$[\stackrel{m}{\SFI{X}{F}{B}}]$}
\noLine
\UIC{$\SFI{X}{F}{\D D'}$}
\noLine
\UIC{$\SFI{X}{F}{A}$}\def\extraVskip{2pt}
\RightLabel{\footnotesize$\supset$E}
\BIC{\usebox{\mypti}}\RightLabel{\footnotesize$\supset$I $(m)$}\UIC{$\SF{X}{F\supset G}{B}$}\DisplayProof 
\ \ =_{\beta}^{\supset} 
$$
$$
=_{\beta}^{\supset}\ \ 
\def\extraVskip{0pt}\def\defaultHypSeparation{\hskip .1em}
\AXC{$\SF{X}{\Gamma}{A}$}
\AXC{$[\stackrel{m}{ \SFI{X}{F}{B}} ]$}
\noLine
\UIC{$\SFI{X}{F}{\D D'}$}
\noLine
\UIC{$\SFI{X}{F}{A}$}
\AXC{$\SFI{X}{\Delta}{B}$}
\noLine
\UIC{$\SFI{X}{\Delta}{\D D'}$}
\noLine
\UIC{$\SFI{X}{\Delta}{A}$}
\noLine\TrinaryInfC{$\SDI{X}{\D D_1}{ A }$}\noLine\UIC{\usebox{\mypti}}\def\extraVskip{2pt}\RightLabel{\footnotesize$\supset$I $(m)$}\UIC{$\SF{X}{F\supset G}{B}$}\DisplayProof
\ \ 
\stackrel{\text{\fbox{i.h.$\!$}}}{=_{\beta}^{2\supset}}\ \ \def\extraVskip{0pt}\def\defaultHypSeparation{\hskip .1em}
\AXC{$\SFI{X}{\Gamma}{A}$}
\noLine
\UIC{$\SFI{X}{\Gamma}{\D D'}$}
\noLine
\UIC{$\SFI{X}{\Gamma}{B}$}
\AXC{$[\stackrel{m}{\SF{X}{F}{B}}]$}
\noLine
\UIC{$\SFI{X}{F}{\D D'}$}
\noLine
\UIC{$\SFI{X}{F}{A}$}
\AXC{$\SFI{X}{\Delta}{B}$}
\noLine\TrinaryInfC{$\SDI{X}{\D D_1}{B}$}\noLine\UIC{$\SFI{X}{G}{B}$}\def\extraVskip{2pt}\RightLabel{\footnotesize$\supset$I $(m)$}\UIC{$\SF{}{F\supset G}{B}$}\DisplayProof 
$$\endgroup
\item if $ \D D\equiv {\small
\AxiomC{$\Gamma_1$}\noLine\UIC{$\mathscr{D}_1$}
\noLine
\UnaryInfC{$F\supset G$}
\AxiomC{$\Gamma_2$}\noLine\UIC{$\mathscr{D}_2$}
\noLine
\UnaryInfC{$F$}
\RightLabel{\footnotesize$\supset$E}
\BinaryInfC{$G$}
\DisplayProof}
$, with $F$ n-$X$ and $G$ p-$X$, then 
\begingroup\makeatletter\def\f@size{10}\check@mathfonts
\begin{lrbox}{\mypti}% Store prooftree in \mypti
\begin{varwidth}{\linewidth}
$\cexpD{X}{G}{\AXC{$A$}\noLine\UIC{$\mathscr{D}'$}\noLine\UIC{$B$}\DP}$
\end{varwidth}
\end{lrbox}
\begin{lrbox}{\mypte}% Store prooftree in \mypti
\begin{varwidth}{\linewidth}
$\cexpD{X}{F\supset G}{\AXC{$A$}\noLine\UIC{$\mathscr{D}'$}\noLine\UIC{$B$}\DP}$
\end{varwidth}
\end{lrbox}
\begin{lrbox}{\mypto}% Store prooftree in \mypti
\begin{varwidth}{\linewidth}
$\cexpD{X}{F}{\AXC{$A$}\noLine\UIC{$\mathscr{D}'$}\noLine\UIC{$B$}\DP}$
\end{varwidth}
\end{lrbox}
\begin{lrbox}{\mypta}% Store prooftree in \mypti
\begin{varwidth}{\linewidth}
$\cexpD{X}{\Gamma_1}{\AXC{$A$}\noLine\UIC{$\mathscr{D}'$}\noLine\UIC{$B$}\DP}$
\end{varwidth}
\end{lrbox}
\begin{lrbox}{\myptu}% Store prooftree in \mypti
\begin{varwidth}{\linewidth}
$\SF{X}{\Gamma_2}{\AXC{$A$}\noLine\UIC{$\mathscr{D}'$}\noLine\UIC{$B$}\DP}$
\end{varwidth}
\end{lrbox}
$$\AXC{$\SF{X}{\Gamma_1}{A}$}
\def\extraVskip{0pt}\AXC{$\SFI{X}{\Delta_{1}}{B}$}
\noLine
\UIC{$\SF{X}{\Delta_{1}}{\D D'}$}
\noLine
\UIC{$\SFI{X}{\Delta_{1}}{A}$}
\noLine\BIC{$\SD{X}{\mathscr{D}_{1}}{A}$}
\noLine\UIC{$\SF{X}{F\supset G}{A}$}
\AXC{$\SF{X}{\Gamma_2}{A}$}
\AXC{$\SFI{X}{\Delta_{2}}{B}$}
\noLine
\UIC{$\SFI{X}{\Delta_{2}}{\D D'}$}
\noLine
\UIC{$\SFI{X}{\Delta_{2}}{A}$}
\noLine\BIC{$\SD{X}{\mathscr{D}_{2}}{A}$}\noLine\UIC{$\SF{X}{F}{A}$}
\def\extraVskip{2pt}\RightLabel{\footnotesize$\supset$E}
\BIC{\SFI{X}{G}{A}}\def\extraVskip{0pt}\noLine\UIC{\SFI{X}{G}{\mathscr{D}'}}\noLine\UIC{\SFI{X}{G}{B}}\DisplayProof 
\ \ \stackrel{\text{\fbox{i.h.$\!$}}}{=_{\beta}^{2\supset}}$$
$$\def\extraVskip{0pt}\stackrel{\text{\fbox{i.h.$\!$}}}{=_{\beta}^{2\supset}}\ \ \AXC{$\SF{X}{\Gamma_1}{A}$}
\AXC{$\SFI{X}{\Delta_{1}}{B}$}
\noLine
\UIC{$\SFI{X}{\Delta_{1}}{\D D'}$}
\noLine
\UIC{$\SFI{X}{\Delta_{1}}{A}$}
\noLine\BIC{$\SD{X}{\mathscr{D}_{1}}{A}$}
\noLine\UIC{$\SF{X}{F\supset G}{A}$}
\AXC{$\SFI{X}{\Gamma_2}{A}$}
\noLine
\UIC{$\SFI{X}{\Gamma_{2}}{\D D'}$}
\noLine
\UIC{$\SFI{X}{\Gamma_{2}}{B}$}
\AXC{$\SF{X}{\Delta_{2}}{B}$}
\noLine\BIC{$\SDI{X}{\mathscr{D}_{2}}{B}$}\noLine\UIC{$\SFI{X}{F}{B}$}
\noLine
\UIC{$\SFI{X}{F}{\D D'}$}
\noLine
\UIC{$\SFI{X}{F}{A}$}
\def\extraVskip{2pt}\RightLabel{\footnotesize$\supset$E}
\BIC{\SFI{X}{G}{A}}\def\extraVskip{0pt}\noLine\UIC{\SFI{X}{G}{\mathscr{D}'}}\noLine\UIC{\SFI{X}{G}{B}}\DisplayProof \ \ =_{\beta}^{2\supset}$$
$$
=_{\beta}^{2\supset}\ \ 
\def\extraVskip{0pt}\AXC{$\SF{X}{\Gamma_1}{A}$}
\AXC{$\SFI{X}{\Delta_{1}}{B}$}
\noLine
\UIC{$\SFI{X}{\Delta_{1}}{\D D'}$}
\noLine
\UIC{$\SFI{X}{\Delta_{1}}{A}$}
\noLine\BIC{$\SD{X}{\mathscr{D}_{1}}{A}$}
\noLine\UIC{$\SF{X}{F\supset G}{A}$}
\AXC{$[\stackrel{n}{\SF{X}{F}{B}}]$}
\noLine
\UIC{$\SF{X}{F}{\D D}$}
\noLine
\UIC{$\SF{X}{F}{A}$}\def\extraVskip{2pt}\RL{$\supset$E}
\BIC{$\SFI{X}{G}{A}\ $}
\def\extraVskip{0pt}\noLine\UIC{\SFI{X}{G}{\mathscr{D}'}}\noLine\UIC{\SFI{X}{G}{B}}
\def\extraVskip{2pt}\RL{$\supset$I $(n)$}
\UIC{$\SF{X}{F\supset G}{B}$}
\def\extraVskip{0pt}\AXC{$\SFI{X}{\Gamma_2}{A}$}
\noLine
\UIC{$\SFI{X}{\Gamma_{2}}{\D D'}$}
\noLine
\UIC{$\SFI{X}{\Gamma_{2}}{B}$}
\AXC{$\SF{X}{\Delta_{2}}{B}$}
\noLine\BIC{$\SD{X}{\mathscr{D}_{2}}{B}$}\noLine\UIC{$\SF{X}{F}{B}$}
\def\extraVskip{2pt}\RightLabel{\footnotesize$\supset$E}
\BIC{\SF{X}{G}{B}}
\DP
\ \ \stackrel{\text{\fbox{i.h.$\!$}}}{=_{\beta}^{2\supset}}$$$$\stackrel{\text{\fbox{i.h.$\!$}}}{=_{\beta}^{2\supset}}\ \ 
\def\extraVskip{0pt}
\AXC{$\SFI{X}{\Gamma_1}{A}$}
\noLine
\UIC{$\SFI{X}{\Gamma_{1}}{\D D'}$}
\noLine
\UIC{$\SFI{X}{\Gamma_{1}}{B}$}
\AXC{$\SF{X}{\Delta_{1}}{B}$}
\noLine\BIC{$\SD{X}{\mathscr{D}_{1}}{B}$}
\noLine\UIC{$\SF{X}{F\supset G}{B}$}
\AXC{$\SFI{X}{\Gamma_2}{A}$}
\noLine
\UIC{$\SFI{X}{\Gamma_{2}}{\D D'}$}
\noLine
\UIC{$\SFI{X}{\Gamma_{2}}{B}$}
\AXC{$\SF{X}{\Delta_{2}}{B}$}
\noLine\BIC{$\SD{X}{\mathscr{D}_{2}}{B}$}\noLine\UIC{$\SF{X}{F}{B}$}
\def\extraVskip{2pt}\RightLabel{\footnotesize$\supset$E}
\BIC{\SF{X}{G}{B}}
\DP
$$
\endgroup
\end{itemize}
\end{proof}

\begin{remark}\label{din-rem}In \cite{Girard1992} it is proved that derivations in $\texttt{NI}$ satisfy a more general property, called \emph{dinaturality}. Dinaturality takes into account formulas in which variables might occur \emph{both} positively and negatively (in categorial terms, dinaturality accounts for both the covariant and the contravariant action of the functors associated to the formulas). This is the reason why, whereas naturality only holds for $X$-safe derivations whose premises and conclusions are either positive or negative, dinaturality holds for all $X$-safe derivations in $\texttt{NI}$.
However, dinaturality is a weaker property than naturality; in particular, it is well-known that two dinatural transformations might fail to compose, see \cite{Bainbridge1990}. 

Observe that the $X$-safety requirement is essential to make the rule $\forall$E (di-)natural (though, in a sense,  $X$-safety makes $\forall$E (di-)natural in a ``trivial'' way) and that neither theorem~\ref{naturality-th} nor its ``dinatural'' generalization can be extended to derivations that are not $X$-safe (a counter-example to dinaturality in $\texttt{NI}^2$ can be found in \cite{Delatail2009}).
\end{remark}

\section{Nested positive formulas and permutative conversions}\label{sec-epsi}

\subsection{The $\varepsilon$-equation}
The left-to-right orientation of the naturality condition discussed in the previous section has the flavour of a permutative conversion, in the sense that the derivation of $B$ from $A,\Delta$ is permuted across the derivation of $C$ from $\Gamma$. A precise connection between naturality and permutative conversions can be spelled out by  introducing the following class of formulas:

\begin{definition}[nested p-$X$ formulas] Let $ \forall\OV Y_{i}$ (for $1\leq i\leq n$) denote finite (possibly empty) sequences of quantifiers. A formula of the form 
\begingroup\makeatletter\def\f@size{10}\check@mathfonts
$$\forall \OV Y_{1}(F_{1}\supset \forall \OV Y_{2}(F_{2}\supset \dots \supset \forall \OV Y_{n}(F_{n}\supset X)\dots ))$$\endgroup
is {\em nested p-$X$} provided that $\forall \OV Y_1 F_{1},\dots, \forall \OV Y_1 \ldots \forall \OV Y_n F_{n}$ are p-$X$  formulas.
\end{definition}

\begin{remark}\label{qsip}
The RP-translation of $A\vee B$ is of the form $\forall X (C\supset D\supset X)$, where $C$ and $D$ are p-$X$. Thus is of the form $\forall X F$, with $F$ nested p-$X$.
\end{remark}

 \begin{remark}
 If $F$ is nested p-$X$ then 
 \begingroup\makeatletter\def\f@size{10}\check@mathfonts$$\SF{X}{F}{A}\equiv  \forall \OV Y_{1}(\SF{X}{F_{1}}{A}\supset \forall \OV Y_{2}(\SF{X}{F_{2}}{A}\supset \dots \supset \forall \OV Y_{n}(\SF{X}{F_{n}}{A}\supset A)\dots ))$$\endgroup
 \end{remark}
We introduce now the following equation, where $F$ is nested p-$X$:
\begingroup\makeatletter\def\f@size{10}\check@mathfonts
\begin{lrbox}{\mypti}% Store prooftree in \mypti
\raisebox{2.88ex}[-0ex][0ex]{\begin{varwidth}{\linewidth}
$\cexpD{X}{F_1}{\AXC{$A$}\noLine\UIC{$\mathscr{D}_2$}\noLine\UIC{$B$}\DP}$
\end{varwidth}}
\end{lrbox}
\begin{lrbox}{\mypto}% Store prooftree in \mypti
\raisebox{2.88ex}[-0ex][0ex]{\begin{varwidth}{\linewidth}
$\cexpD{X}{F_n}{\AXC{$A$}\noLine\UIC{$\mathscr{D}_2$}\noLine\UIC{$B$}\DP}$
\end{varwidth}}
\end{lrbox}
\begin{equation*}
\def\defaultHypSeparation{\hskip .05in}\AXC{$\mathscr{D}_1$}
\noLine
\UnaryInfC{$\forall X F$}\RightLabel{\footnotesize$\forall $E}\UIC{$\SF{X}{\forall \OV Y_{1}(F_{1}\supset \forall \OV Y_{2}(F_{2}\supset \dots \supset \forall \OV Y_{n}(F_{n}\supset X)\dots ))}{A}$}\doubleLine\RightLabel{\footnotesize$\forall $E}\UIC{$\SF{X}{F_{1}\supset \forall \OV Y_{2}(F_{2}\supset \dots \supset \forall \OV Y_{n}(F_{n}\supset X)\dots )}{A}$}
\AxiomC{$\SF{X}{F_{1}}{A}$}
\RightLabel{\footnotesize$\supset$E}
\BinaryInfC{$\SF{X}{\forall \OV Y_{2}(F_{2}\supset \dots \supset \forall \OV Y_{n}(F_{n}\supset X)\dots )}{A}$}
\noLine
\UIC{$\ddots$}
\noLine
\UIC{$\ $}
\noLine
\UIC{$\SF{X}{\forall \OV Y_{n}(F_{n}\supset X)}{A}$}\doubleLine\RightLabel{\footnotesize$\forall$E}\UIC{$\SF{X}{F_{n}\supset X}{A}$}
\AXC{$\SF{X}{F_{n}}{A}$}
\RightLabel{\footnotesize$\supset$E}\def\defaultHypSeparation{\hskip -.1in}
\BIC{$A$}
\noLine
\UnaryInfC{$\mathscr{D}_2$}
\noLine
\UnaryInfC{$B$}
\DisplayProof  \end{equation*}
\begin{equation*}\text{\Large$=_{\varepsilon}$} \tag{$\varepsilon$}\label{epsilonew}
\end{equation*}
\begin{equation*}
\def\defaultHypSeparation{\hskip .05in}\AXC{$\mathscr{D}_1$}
\noLine
\UnaryInfC{$\forall X F$}\RightLabel{\footnotesize$\forall $E}\UIC{$\SF{X}{\forall \OV Y_{1}(F_{1}\supset \forall \OV Y_{2}(F_{2}\supset \dots \supset \forall \OV Y_{n}(F_{n}\supset X)\dots ))}{B}$}\doubleLine\RightLabel{\footnotesize$\forall$E}\UIC{$\SF{X}{F_{1}\supset \forall \OV Y_{2}(F_{2}\supset \dots \supset \forall \OV Y_{n}(F_{n}\supset X)\dots )}{B}$}\AxiomC{\usebox{\mypti}}
  \RightLabel{\footnotesize$\supset$E}
  \BIC{$\SF{X}{\forall \OV Y_{2}(F_{2}\supset \dots \supset \forall \OV Y_{n}(F_{n}\supset X)\dots )}{B}$} \noLine
 \UIC{$\ddots$}\noLine
\UIC{$\SF{X}{\forall \OV Y_{n}(F_{n}\supset X)}{B}$}\doubleLine\RightLabel{\footnotesize$\forall$E}\UIC{$\SF{X}{F_{n}\supset X}{B}$}\AxiomC{\usebox{\mypto}}
 \RightLabel{\footnotesize$\supset$E}\def\defaultHypSeparation{\hskip -.3in}
 \BinaryInfC{$B$}\DP
\end{equation*}%\end{equation*}%}
\endgroup
We will call $=_\varepsilon$ the equivalence over \Nd derivations generated by the smallest congruence relation closed under reflexivity, symmetry, transitivity and the schema \eqref{epsilonew}.

As an immediate consequence of theorem \ref{naturality-th} we obtain a proposition analogous to \ref{etaclose} and \ref{gammaclose} for the instances of \eqref{epsilonew} in which the derivation $\D D_1$ of $\forall X F$ is closed (we call these instances m-closed and we indicate with $=_{\varepsilon^{mc}}$ the equivalence relation induced by them):

\begin{proposition}\label{elimnat}
If $\D D',\D D''$ are $X$-safe and $\D D' =_{\varepsilon^{mc}} \D D''$, then $\D D'
=_{\beta}^{2\supset} \D D''$.
\end{proposition}

\begin{proof}Consider an instance of \eqref{epsilonew} in which ${\D D_1}$ is {\em closed}. Call $\D D'$ and $\mathscr{D}''$ the left-hand side and right-hand side of this instance of \eqref{epsilonew}. We show $\mathscr{D}'=^{2\supset}_{\beta} \mathscr{D}''$.

Since $\mathscr{D}_1$ is closed, it $\beta$-reduces to a derivation ${\mathscr{D}_1^\sharp}\equiv {\small \AXC{${\mathscr{D}_1^\sharp}'$}\noLine\UIC{$F$}\RightLabel{\footnotesize$\forall$I }\UIC{$\forall X F$}\DP}$. Thus the left-hand side and right-hand side $\beta$-reduce to the following two derivations (by first reducing $\mathscr{D}_1$ to ${\mathscr{D}_1^\sharp}$ and then by getting rid of the resulting $\beta$-redex having $\forall X F$ as maximal formula):
\begingroup\makeatletter\def\f@size{10}\check@mathfonts
\begin{lrbox}{\mypti}% Store prooftree in \mypti
\raisebox{2.88ex}[-0ex][0ex]{\begin{varwidth}{\linewidth}
$\cexpD{X}{F_1}{\AXC{$A$}\noLine\UIC{$\mathscr{D}_2$}\noLine\UIC{$B$}\DP}$
\end{varwidth}}
\end{lrbox}
\begin{lrbox}{\mypto}% Store prooftree in \mypti
\raisebox{2.88ex}[-0ex][0ex]{\begin{varwidth}{\linewidth}
$\cexpD{X}{F_n}{\AXC{$A$}\noLine\UIC{$\mathscr{D}_2$}\noLine\UIC{$B$}\DP}$
\end{varwidth}}
\end{lrbox}
\begin{equation*}
\def\defaultHypSeparation{\hskip .05in}\AXC{$\SD{X}{{\mathscr{D}_1^\sharp}'}{A}$}
\noLine
\UnaryInfC{$\SF{X}{F}{A}$}\doubleLine\RightLabel{\footnotesize$\forall $E}\UIC{$\SF{X}{F_{1}\supset \forall \OV Y_{2}(F_{2}\supset \dots \supset \forall \OV Y_{n}(F_{n}\supset X)\dots )}{A}$}
\AxiomC{$\SF{X}{F_{1}}{A}$}
\RightLabel{\footnotesize$\supset$E}
\BinaryInfC{$\SF{X}{\forall \OV Y_{2}(F_{2}\supset \dots \supset \forall \OV Y_{n}(F_{n}\supset X)\dots )}{A}$}
\noLine
\UIC{$\ddots$}
\noLine
\UIC{$\ $}
\noLine
\UIC{$\SF{X}{\forall \OV Y_{n}(F_{n}\supset X)}{A}$}\doubleLine\RightLabel{\footnotesize$\forall$E}\UIC{$\SF{X}{F_{n}\supset X}{A}$}
\AXC{$\SF{X}{F_{n}}{A}$}
\RightLabel{\footnotesize$\supset$E}\insertBetweenHyps{\hspace{-1.2cm}}
\BIC{$A$}
\noLine
\UnaryInfC{$\mathscr{D}_2$}
\noLine
\UnaryInfC{$B$}
\DisplayProof\end{equation*}
\begin{equation*}
\qquad\def\defaultHypSeparation{\hskip .05in}\AXC{$\SD{X}{{\mathscr{D}_1^\sharp}'}{B}$}
\noLine
\UnaryInfC{$\SF{X}{F}{B}$}\doubleLine\RightLabel{\footnotesize$\forall$E}\UIC{$\SF{X}{F_{1}\supset \forall \OV Y_{2}(F_{2}\supset \dots \supset \forall \OV Y_{n}(F_{n}\supset X)\dots )}{B}$}\AxiomC{\usebox{\mypti}}
  \RightLabel{\footnotesize$\supset$E}
  \BIC{$\SF{X}{\forall \OV Y_{2}(F_{2}\supset \dots \supset \forall \OV Y_{n}(F_{n}\supset X)\dots )}{B}$}\
 \noLine
 \UIC{}
 \noLine
 \UIC{$\ddots$}\noLine
\UIC{$\SF{X}{\forall \OV Y_{n}(F_{n}\supset X)}{B}$}\doubleLine\RightLabel{\footnotesize$\forall$E}\UIC{$\SF{X}{F_{n}\supset X}{B}$}\AxiomC{\usebox{\mypto}}
 \insertBetweenHyps{\hspace{-1.3cm}}\RightLabel{\footnotesize$\supset$E}
 \BinaryInfC{$B$}\DP
\end{equation*}
\endgroup
Their $\beta^{2\supset}$-equality follows from fact that, since $\D D_1$ is closed and $X$-safe, so is ${\D D_1^\sharp}'$, and hence the derivation 
\begingroup\makeatletter\def\f@size{10}\check@mathfonts
$$\AXC{${\mathscr{D}_1^\sharp}'$}
\noLine
\UnaryInfC{$F$}\doubleLine\RightLabel{\footnotesize$\forall $E}\UIC{$F_{1}\supset \forall \OV Y_{2}(F_{2}\supset \dots \supset \forall \OV Y_{n}(F_{n}\supset X)\dots )$}
\AxiomC{$F_{1}$}
\RightLabel{\footnotesize$\supset$E}
\BinaryInfC{$\forall \OV Y_{2}(F_{2}\supset \dots \supset \forall \OV Y_{n}(F_{n}\supset X)\dots )$}
\noLine
\UIC{$\ddots$}
\noLine
\UIC{$\ $}
\noLine
\UIC{$\forall \OV Y_{n}(F_{n}\supset X)$}\doubleLine\RightLabel{\footnotesize$\forall$E}\UIC{$F_{n}\supset X$}
\AXC{$F_{n}$}
\RightLabel{\footnotesize$\supset$E}
\BIC{$X$}
\DisplayProof
$$\endgroup
is p-$X$, $X$-safe and thus $X$-natural by theorem \ref{naturality-th}.
\end{proof}

Whereas m-closed instances of \eqref{epsilonew} are included in $=_{\beta}^{2\supset}$, there are instances of \eqref{epsilonew} which are not. This means that the equational theory on \Nd-derivations induced by \eqref{epsilonew} together with \eqref{beta2} and \eqref{betaimp} (we indicate it with $=_{\beta\varepsilon}^{2\supset}$) is a strict extension of the one induced by \eqref{beta2} and \eqref{betaimp} alone. Actually, the same is true if one considers the extension $=_{\beta\eta\varepsilon}^{2\supset}$ of $=_{\beta\eta}^{2\supset}$.

\begin{proposition}\label{epsf}
The equational theory $=_{\beta\eta\varepsilon}^{2\supset}$ strictly extends $=_{\beta\eta}^{2\supset}$.\end{proposition}

\begin{proof}  
For any derivation $\mathscr{D}'$ of $B$ from $A$ we have that 
\begingroup\makeatletter\def\f@size{10}\check@mathfonts
\begin{lrbox}{\mypti}% Store prooftree in \mypti
\raisebox{2.88ex}[-0ex][0ex]{
\begin{varwidth}{\linewidth}
$\cexpD{X}{X}{\AXC{$A$}\noLine\UIC{$\mathscr{D}'$}\noLine\UIC{$B$}\DP}$
\end{varwidth}
}
\end{lrbox}
\begin{equation*}
\def\extraVskip{2pt}\AxiomC{$\forall X (X\supset X)$}\RightLabel{\footnotesize$\forall $E}\UIC{$A\supset A$}
\AxiomC{$A$}
\RightLabel{\footnotesize$\supset$E}
\BinaryInfC{$A$}
\noLine
\UnaryInfC{$\D D'$}
\noLine
\UnaryInfC{$B$}
\DisplayProof
\  \equiv\ 
\def\extraVskip{2pt}\AxiomC{$\forall X (X\supset X)$}\RightLabel{\footnotesize$\forall $E}\UIC{$\SF{X}{X\supset X}{A}$}\AXC{$\SF{X}{X}{A}$}\RightLabel{\footnotesize$\supset $E}\BIC{$A$}\noLine\UIC{$\mathscr{D}'$}\noLine\UIC{$B$}\DP
\   =_{\varepsilon}
\end{equation*}
\begin{equation*}
=_{\varepsilon}\   \def\extraVskip{2pt}\AXC{$\forall X (X\supset X)$}\RightLabel{\footnotesize$\forall $E}\UIC{$\SF{X}{X\supset X}{B}$}\AXC{\usebox{\mypti}}\BinaryInfC{$B$}\DP\  \equiv\ \ \def\extraVskip{2pt}\AxiomC{$\forall X (X\supset X)$}\RightLabel{\footnotesize$\forall $E}\UIC{$B\supset B$}\AxiomC{$A$}\def\extraVskip{2pt}
\noLine
\UnaryInfC{$\D D'$}
\noLine
\UnaryInfC{$B$}
\RightLabel{\footnotesize$\supset$E}
\BinaryInfC{$B$}
\DisplayProof
\end{equation*}
\endgroup
However, whenever $\mathscr{D}'$ is $\beta\eta$-normal, 
the two member of the instance of \eqref{epsilonew} just considered are distinct $\beta\eta$-normal derivations and therefore (by the Church-Rosser property of the rewriting relation induced by $\beta$- and $\eta$-reductions in \Nd) they are not $\beta\eta$-equivalent. 
\end{proof}

\begin{remark}
It is worth stressing that this extension of $=_{\beta\eta}^{2\supset}$ is  \emph{consistent}: there exist derivations of the same conclusion from the same undischarged assumptions which are not identified by $=_{\beta\eta\varepsilon}^{2\supset}$. This follows from the fact \cite{Bainbridge1990, Freyd88,Girard1992} that there are models of System $\B F$ satisfying the ``dinatural'' generalization of  \eqref{epsilonew} (see remark \ref{din-rem}).

Observe that via Curry-Howard, one can consider $=_{\beta\eta\varepsilon}$ as an equivalence over $\lambda$-terms. Whereas \eqref{epsilonew} consistently extends $\beta\eta$-equality in System $\B F$, in the untyped case $=_{\beta\eta}$ is maximal (as a consequence of B\"ohm's theorem) and thus $=_{\beta\eta\varepsilon}$ is inconsistent over untyped lambda terms (i.e.~$s=_{\beta\eta\varepsilon} t$ for all terms $s,t$). 
\end{remark}

\subsection{RP-translation and $\varepsilon$-equations}\label{rpdisj}
As observed in section \ref{sec-rp}, the RP-translation {\em does not} map $=_{\gamma_g}^{\vee}$ into either $=_{\beta}^{2\supset}$ or $=_{\beta\eta}^{2\supset}$, in the sense that $\mathscr{D}_1=_{\gamma_g}^{\vee}\mathscr{D}_2$ implies neither $\mathscr{D}^*_1=_{\beta}^{2\supset}\mathscr{D}^*_2$ nor $\mathscr{D}_1^*=_{\beta\eta}^{2\supset}\mathscr{D}_2^*$. However, the RP-translation {\em does} map $=_{\gamma_g}^{\vee}$ into $=_{\beta\eta\varepsilon}^{2\supset}$, in fact into $=_{\beta\varepsilon}$ alone. More precisely,

\begin{proposition}[$=_{\gamma_g}^{\vee}\ \stackrel{*}{\mapsto}\ =_{\beta\varepsilon}$]\label{conv}
Let $\D D'$ and $\mathscr{D}''$ be, respectively, the left-hand side and right-hand side of \eqref{gammaorG}. 
One has $\mathscr{D}'^{*}=_{\beta\varepsilon} \mathscr{D}''^{*}$.
\end{proposition}

\begin{proof}
Since $(A\vee B)^*\equiv \forall X ((A^*\supset X)\supset (B^*\supset X)\supset X)$, by proposition \ref{elimnat} and remark \ref{epsfree} we have that
\begingroup\makeatletter\def\f@size{10}\check@mathfonts
\begin{lrbox}{\mypto}% Store prooftree in \mypti
\raisebox{2.88ex}[7.7ex][0ex]{\begin{varwidth}{\linewidth}
$\cexpD{X}{A^{*}\supset X}{\!\!\AXC{$C^*$}\noLine\UIC{$\mathscr{D}_3$}\noLine\UIC{$D^*$}\DP\!\!}$
\end{varwidth}}
\end{lrbox}
\begin{lrbox}{\mypta}% Store prooftree in \mypti
\raisebox{2.88ex}[7.7ex][0ex]{\begin{varwidth}{\linewidth}
$\cexpD{X}{B\supset X}{\!\!\AXC{$C^*$}\noLine\UIC{$\mathscr{D}_3$}\noLine\UIC{$D^*$}\DP\!\!}$\end{varwidth}}
\end{lrbox}
\begin{equation*}
\mathscr{D}'^*\ \equiv\ \def\defaultHypSeparation{\hskip .01in}\AxiomC{$\mathscr{D}^*$}\noLine\UIC{$(A\vee B)^*$}
\RightLabel{\footnotesize$\forall$E}
\UnaryInfC{$(A^*\supset C^*)\supset (B^*\supset C^*)\supset C^*$}
\AxiomC{$\stackrel{n_1}{[A^*]}$}\noLine\UIC{$\mathscr{D}_1^*$}\noLine\UIC{$C^*$}\RightLabel{\footnotesize$\supset$I $(n_1)$}\UIC{$A^*\supset C^*$}
\RightLabel{\footnotesize$\supset$E}
\BinaryInfC{$(B^*\supset C^*)\supset C^*$}
\AxiomC{$\stackrel{n_2}{[B^*]}$}\noLine\UIC{$\mathscr{D}_2^*$}\noLine\UIC{$C^*$}\RightLabel{\footnotesize$\supset$I $(n_2)$}\UIC{$B^*\supset C^*$}
\RightLabel{\footnotesize$\supset$E}
\BinaryInfC{$C^*$}
\noLine
\UnaryInfC{$\D D_3^*$}
\noLine
\UnaryInfC{$D^*$}
\DisplayProof\end{equation*}
$$\equiv$$
\begin{equation*}\def\extraVskip{2pt}\def\defaultHypSeparation{\hskip .01in}\AxiomC{$\mathscr{D}^*$}\noLine\UIC{$\forall X ((A^*\supset X)\supset (B^*\supset X)\supset X)$}
\RightLabel{\footnotesize$\forall$E}
\UnaryInfC{$\SF{X}{A^*\supset X}{C^*\!}\supset \SF{X}{B^*\supset X}{C^*\!}\supset C^*$}
\AxiomC{$\stackrel{n_1}{[A^*]}$}\noLine\UIC{$\mathscr{D}_1^*$}\noLine\UIC{$C^*$}\RightLabel{\footnotesize$\supset$I $(n_1)$}\UIC{$\SF{X}{A^*\supset X}{C^*\!}$}
\RightLabel{\footnotesize$\supset$E}
\BinaryInfC{$\SF{X}{B^*\supset X}{C^*\!}\supset C^*$}
\AxiomC{$\stackrel{n_2}{[B^*]}$}\noLine\UIC{$\mathscr{D}_2^*$}\noLine\UIC{$C^*$}\RightLabel{\footnotesize$\supset$I $(n_2)$}\UIC{$\SF{X}{B^*\supset X}{C^*\!}$}
\RightLabel{\footnotesize$\supset$E}\def\defaultHypSeparation{\hskip -.23in}
\BinaryInfC{$C^*$}
\noLine
\UnaryInfC{$\D D_3^*$}
\noLine
\UnaryInfC{$D^*$}
\DisplayProof
\end{equation*}
\nopagebreak[5]
$$\ =_\varepsilon$$
\begin{equation*}
\def\defaultHypSeparation{\hskip .01in}\def\extraVskip{2pt}\AxiomC{$\mathscr{D}^*$}\noLine\UIC{$\forall X ((A^*\supset X)\supset (B^*\supset X)\supset X)$}\RightLabel{\footnotesize$\forall $E}\UIC{$\SF{X}{A^*\supset X}{D}\supset \SF{X}{B^*\supset X}{D}\supset D$}\AxiomC{$\stackrel{n_1}{[A^*]}$}\noLine\UIC{$\mathscr{D}_1^*$}\noLine\UIC{$C^*$}\RightLabel{\footnotesize$\supset$I $(n_1)$}\UIC{\usebox{\mypto}}\RightLabel{\footnotesize$\supset$E}\BIC{$\SF{X}{B^*\supset X}{D}\supset D$}\AxiomC{$\stackrel{n_2}{[B^*]}$}\noLine\UIC{$\mathscr{D}_2^*$}\noLine\UIC{$C^*$}\RightLabel{\footnotesize$\supset$I $(n_2)$}\UIC{\usebox{\mypta}}\insertBetweenHyps{\hspace{-.5cm}}\RightLabel{\footnotesize$\supset$E}\BIC{$D$}\DP
\end{equation*}
$$\equiv$$
\begin{equation*}
\def\defaultHypSeparation{\hskip .01in}\AxiomC{$\mathscr{D}^*$}\noLine\UIC{$(A\vee B)^*$}
\RightLabel{\footnotesize$\forall$E}
\UnaryInfC{$(A^*\supset D)\supset (B^*\supset D)\supset D$}
\AxiomC{$\stackrel{n_1}{[A^*]}$}\noLine\UIC{$\mathscr{D}_1^*$}\noLine\UIC{$C^*$}\RightLabel{\footnotesize$\supset$I $(n_1)$}\UIC{$A^*\supset C$}\AxiomC{$\stackrel{m_1}{A^{*}}$}\RightLabel{\footnotesize$\supset$E}
\BIC{$C$}
\noLine
\UnaryInfC{$\D D_3^*$}
\noLine
\UnaryInfC{$D$}
\RightLabel{\footnotesize$\supset$I $(m_1)$}
\UnaryInfC{$A^*\supset D$}
\RightLabel{\footnotesize$\supset$E}\insertBetweenHyps{\hspace{-.7cm}}
\BinaryInfC{$(B^*\supset D)\supset D$}
\AxiomC{$\stackrel{n_2}{[B^*]}$}\noLine\UIC{$\mathscr{D}_2^*$}\noLine\UIC{$C^*$}\RightLabel{\footnotesize$\supset$I $(n_2)$}\UIC{$B^*\supset C^*$}\AxiomC{$\stackrel{m_2}{B^{*}}$}\RightLabel{\footnotesize$\supset$E}
\BIC{$C$}
\noLine
\UnaryInfC{$\D D_3^*$}
\noLine
\UnaryInfC{$D$}
\RightLabel{\footnotesize$\supset$I $(m_2)$}
\UnaryInfC{$B^*\supset D$}
\RightLabel{\footnotesize$\supset$E}
\BinaryInfC{$D$}
\DisplayProof
\end{equation*}\bigskip
$$\ =_{\beta}^\supset \D D''^*$$
\endgroup
\end{proof}

\begin{remark}
By inspecting the proof of proposition \ref{conv} it is clear that for m-closed instances of $=_{\gamma_g}^{\vee}$, we have that $\mathscr{D}'=_{\gamma_g^{mc}}^\vee \mathscr{D}''$ implies $\mathscr{D}'^*=_{\varepsilon^{mc}\beta}^{\supset} \mathscr{D}''^*$. Thus, proposition \ref{conv} together with theorem \ref{naturality-th} provides an alternative way to establish corollary \ref{mc-gammapres}. 
\end{remark}

Not only does the RP-translation map $=_{\gamma_g}^{\vee}$ into $=_{\varepsilon}$, but it also maps $=_{\eta}^{\vee}$ into $=_{\beta\eta\varepsilon}^{2\supset}$. More precisely,

\begin{proposition}[$=_{\eta}^{\vee}\ \stackrel{*}{\mapsto}\ =_{\beta\eta\varepsilon}^{2\supset}$]\label{etaa}
Let 
$\D D'$ and $\mathscr{D}''$ be, respectively, the left-hand side and right-hand side of \eqref{etaor}. 
One has $\mathscr{D}'^{*}=_{\beta\eta\varepsilon}^{2\supset} \mathscr{D}''^{*}$.
\end{proposition}

\begin{proof}
We will use $A^*\curlyvee B^*$ as a shorthand for $(A^{*}\supset X)\supset (B^{*}\supset X)\supset X$. Thus 
\begin{itemize}
 \item $\SF{X}{A^{*}\curlyvee B^{*}}{A^*\vee B^*}\equiv (A^{*}\supset (A\vee B)^*)\supset (B^{*}\supset (A\vee B)^*)\supset (A\vee B)^*$
 \item$\SF{X}{A^{*}\curlyvee B^{*}}{X}\equiv A^{*}\curlyvee B^{*}$
\end{itemize}
  Moreover we set 
\begingroup\makeatletter\def\f@size{10}\check@mathfonts 
 $$\AXC{$(A\vee B)^*$}\noLine\UIC{$\mathscr{D}_2$}\noLine\UIC{$X$}\DP\ \  \equiv\ \ \ \  \def\defaultHypSeparation{\hskip .01in}\AXC{$(A\vee B)^*$}\RightLabel{\footnotesize$\forall $E}\UIC{$A^*\curlyvee B^*$}\AXC{$A^*\supset X$}\RightLabel{\footnotesize$\supset $E}\BIC{$(B^*\supset X)\supset X$}\AXC{$B^*\supset X$}\RightLabel{\footnotesize$\supset $E}\BIC{$X$}\DP$$\endgroup
 (thereby leaving the assumptions $A^{*}\supset X$ and $B^{*}\supset X$ implicit) and
\begingroup\makeatletter\def\f@size{10}\check@mathfonts 
$$\mathscr{D}_A \ \ \equiv  \ \ 
\AxiomC{$\stackrel{n_1}{A^{*}\supset X}$}
\AxiomC{$\stackrel{n_2}{A^{*}}$}
\RightLabel{\footnotesize$\supset$E}
\BinaryInfC{$X$}
\RightLabel{\footnotesize$\supset$I}
\UnaryInfC{$(B^{*}\supset X)\supset X$}
\RightLabel{\footnotesize$\supset$I $(n_1)$}
\UnaryInfC{$A^{*}\curlyvee B^{*}$}
\RightLabel{\footnotesize$\forall I$}
\UnaryInfC{$(A\lor B)^{*}$}
\RightLabel{\footnotesize$\supset$I $(n_2)$}
\UnaryInfC{$A^{*}\supset (A\lor B)^{*}$}\DP\ \  \equiv  \ \ 
\AxiomC{$\stackrel{n_1}{A^{*}\supset X}$}
\AxiomC{$\stackrel{n_2}{A^{*}}$}
\RightLabel{\footnotesize$\supset$E}
\BinaryInfC{$X$}
\RightLabel{\footnotesize$\supset$I}
\UnaryInfC{$(B^{*}\supset X)\supset X$}
\RightLabel{\footnotesize$\supset$I $(n_1)$}
\UnaryInfC{$A^{*}\curlyvee B^{*}$}
\RightLabel{\footnotesize$\forall I$}
\UnaryInfC{$(A\lor B)^{*}$}
\RightLabel{\footnotesize$\supset$I $(n_2)$}
\UnaryInfC{$\SF{X}{A^{*}\supset X}{A^*\vee B^*}$}\DP
$$
$$\mathscr{D}_B \ \ \equiv  \ \ 
\AxiomC{$\stackrel{m_1}{B^{*}\supset X}$}
\AxiomC{$\stackrel{m_2}{B^{*}}$}
\RightLabel{\footnotesize$\supset$E}
\BinaryInfC{$X$}
\RightLabel{\footnotesize$\supset$I $(m_1)$}
\UnaryInfC{$(B^{*}\supset X)\supset X$}
\RightLabel{\footnotesize$\supset$I}
\UnaryInfC{$A^{*}\curlyvee B^{*}$}
\RightLabel{\footnotesize$\forall I$}
\UnaryInfC{$(A\lor B)^{*}$}
\RightLabel{\footnotesize$\supset$I $(m_2)$}
\UnaryInfC{$B^{*}\supset (A\lor B)^{*}$}\DP\ \ \equiv \ \ 
\AxiomC{$\stackrel{m_1}{B^{*}\supset X}$}
\AxiomC{$\stackrel{m_2}{B^{*}}$}
\RightLabel{\footnotesize$\supset$E}
\BinaryInfC{$X$}
\RightLabel{\footnotesize$\supset$I $(m_1)$}
\UnaryInfC{$(B^{*}\supset X)\supset X$}
\RightLabel{\footnotesize$\supset$I}
\UnaryInfC{$A^{*}\curlyvee B^{*}$}
\RightLabel{\footnotesize$\forall I$}
\UnaryInfC{$(A\lor B)^{*}$}
\RightLabel{\footnotesize$\supset$I $(m_2)$}
\UnaryInfC{$\SF{X}{B^{*}\supset X}{A^*\vee B^*}$}\DP
$$\endgroup

By proposition \ref{elimnat} and remark \ref{epsfree} we have that $\mathscr{D}'^*\equiv $
\begingroup\makeatletter\def\f@size{10}\check@mathfonts 
$$\equiv \ \AxiomC{$\D D^{*}$}
\noLine
\UnaryInfC{$(A\lor B)^{*}$}
\RightLabel{\footnotesize$\forall E$}
\UnaryInfC{$\SF{X}{A^{*}\curlyvee B^{*}}{A^*\vee B^*}$}
\AxiomC{$\mathscr{D}_A$}\noLine
\UnaryInfC{$A^{*}\supset (A\lor B)^{*}$}
\RightLabel{\footnotesize$\supset$E}
\BinaryInfC{$(B^*\supset (A\lor B)^{*})\supset (A\lor B)^{*}$}
\AxiomC{$\mathscr{D}_B$}\noLine\UnaryInfC{$B^{*}\supset (A\lor B)^{*}$}
\RightLabel{\footnotesize$\supset$E}
\BinaryInfC{$(A\lor B)^{*}$}
\DisplayProof
\ =_{\eta}^{2\supset}$$
\begin{lrbox}{\mypta}% Store prooftree in \mypti
\begin{varwidth}{\linewidth}
$\SF{X}{X}{\AXC{$(A\vee B)^*$}\noLine\UIC{$\mathscr{D}_2$}\noLine\UIC{$X$}\DP}$
\end{varwidth}
\end{lrbox}
\begin{lrbox}{\mypto}% Store prooftree in \mypti
\raisebox{3.25ex}[7.7ex][0ex]{\begin{varwidth}{\linewidth}
$\SF{X}{A^*\supset X}{\!\!\!\!\AXC{$(A\vee B)^*$}\noLine\UIC{$\mathscr{D}_2$}\noLine\UIC{$X$}\DP\!\!\!\!}$
\end{varwidth}}
\end{lrbox}
\begin{lrbox}{\mypti}% Store prooftree in \mypti
\raisebox{3.25ex}[7.7ex][0ex]{\begin{varwidth}{\linewidth}
$\SF{X}{B^*\supset X}{\AXC{$(A\vee B)^*$}\noLine\UIC{$\mathscr{D}_2$}\noLine\UIC{$X$}\DP}$
\end{varwidth}}
\end{lrbox}
%
% \begin{lrbox}{\myptu}% Store prooftree in \mypti
% \begin{varwidth}{1.8cm}
% $\phantom{A}$
% 
% \vspace{-1.8cm}
% $\AxiomC{$\mathscr{D}_B$}\noLine\UnaryInfC{\usebox{\mypti}}\DP$
% \end{varwidth}
% \end{lrbox}
%
$$=_{\eta}^{2\supset}\ \def\defaultHypSeparation{\hskip .1in}\AxiomC{$\D D^{*}$}
\noLine
\UnaryInfC{$(A\lor B)^{*}$}
\RightLabel{\footnotesize$\forall E$}
\UnaryInfC{$\SF{X}{A^{*}\curlyvee B^{*}}{A^*\vee B^*}$}
  \AxiomC{$\mathscr{D}_A$}\noLine\UnaryInfC{$\SF{X}{A^*\supset X}{(A\lor B)^{*}}$}
  \BIC{$\SF{X}{B^*\supset X}{(A\lor B)^{*}}\supset (A\lor B)^{*}$}
  \AxiomC{$\mathscr{D}_B$}\noLine
  \UnaryInfC{$\SF{X}{B^*\supset X}{(A\lor B)^{*}}$}
\BIC{$(A\lor B)^{*}$}\noLine\UnaryInfC{$\D D_2$}\noLine\UIC{$X$}\RightLabel{\footnotesize$\supset$I}\UIC{$(B^*\supset X)\supset X$}\RightLabel{\footnotesize$\supset$I}\UIC{$A^*\curlyvee B^*$}\RightLabel{\footnotesize$\forall$I}\UIC{$(A\vee B)^*$}\DP \ =_{\varepsilon}$$
$$=_{\varepsilon} \ \AxiomC{$\D D^{*}$}
\noLine
\UnaryInfC{$(A\lor B)^{*}$}
\RightLabel{\footnotesize$\forall E$}
\UnaryInfC{$\SF{X}{A^{*}\curlyvee B^{*}}{X}$}
  \AxiomC{$\mathscr{D}_A$}\noLine
  \UnaryInfC{\usebox{\mypto}}\BIC{$\SF{X}{B^*\supset X}{X}$}
 \AxiomC{$\mathscr{D}_B$}\noLine\UnaryInfC{\usebox{\mypti}}
\BIC{$X$}\RightLabel{\footnotesize$\supset$I}\UIC{$(B^*\supset X)\supset X$}\RightLabel{\footnotesize$\supset$I}\UIC{$A^*\curlyvee B^*$}\RightLabel{\footnotesize$\forall$I}\UIC{$(A\vee B)^*$}\DP\ \equiv$$
$$\equiv \ \def\defaultHypSeparation{\hskip .0in}
\AxiomC{$\D D^{*}$}
\noLine
\UnaryInfC{$(A\lor B)^{*}$}
%\RightLabel{\footnotesize$\forall E$}
\UnaryInfC{$A^{*}\curlyvee B^{*}$}
\AxiomC{$\mathscr{D}_A$}\noLine\UIC{$A^*\supset (A\lor B)^{*}$}
\AXC{$\stackrel{n_3}{A^*}$}
%\RightLabel{\footnotesize$\supset $E}
\BIC{$(A\lor B)^{*}$}
%\RightLabel{\footnotesize$\forall E$}
\UnaryInfC{$A^*\curlyvee B^*$}
\AxiomC{$\stackrel{o_2}{A^{*}\supset X}$}
%\RightLabel{\footnotesize$\supset$E}
\insertBetweenHyps{\hspace{-.6cm}}
\BinaryInfC{$(B^{*}\supset X)\supset X$}
\AxiomC{$\stackrel{o_1}{B^{*}\supset X}$}
%\RightLabel{\footnotesize$\supset$E}
\insertBetweenHyps{\hspace{-.1cm}}
\BinaryInfC{$X$}
\RightLabel{\footnotesize%$\supset$I 
$(n_3)$}
\UnaryInfC{$A^{*}\supset X$}
%\RightLabel{\footnotesize$\supset$E}
\insertBetweenHyps{\hspace{-1.3cm}}
\BinaryInfC{$(B^*\supset X)\supset X$}
\AxiomC{$\mathscr{D}_B$}\noLine
\UIC{$B^*\supset (A\lor B)^{*}$}
\AXC{$\stackrel{m_3}{B^*}$}
%\RightLabel{\footnotesize$\supset$E}
\BIC{$A\lor B)^{*}$}
%\RightLabel{\footnotesize$\forall$E}
\UnaryInfC{$A^*\curlyvee B^*$}
\AxiomC{$\stackrel{o_2}{A^{*}\supset X}$}
%\RightLabel{\footnotesize$\supset$E}
\insertBetweenHyps{\hspace{-.6cm}}
\BinaryInfC{$(B^{*}\supset X)\supset X$}
\AxiomC{$\stackrel{o_1}{B^{*}\supset X}$}
%\RightLabel{\footnotesize$\supset$E}
\insertBetweenHyps{\hspace{-.1cm}}
\BinaryInfC{$X$}
\RightLabel{\footnotesize%$\supset$I 
$(m_3)$}
\UnaryInfC{$B^{*}\supset X$}
%\RightLabel{\footnotesize$\supset$E}
\insertBetweenHyps{\hspace{-.7cm}}
\BinaryInfC{$X$}
\RightLabel{\footnotesize%$\supset$I 
$(o_1)$}
\UnaryInfC{$(B^{*}\supset X)\supset X$}
\RightLabel{\footnotesize%$\supset$I 
$(o_2)$}
\UnaryInfC{$A^*\curlyvee B^*$}
\DisplayProof\ =_{\beta}^{2\supset}
$$
$$=_{\beta}^{2\supset}\ 
\AxiomC{$\D D^{*}$}
\noLine
\UnaryInfC{$(A\lor B)^{*}$}
\RightLabel{\footnotesize$\forall E$}
\UnaryInfC{$A^{*}\curlyvee B^{*}$}\AxiomC{$\stackrel{o_2}{A^{*}\supset X}$}
\AxiomC{$\stackrel{n_3}{A^{*}}$}
\RightLabel{\footnotesize$\supset$E}
\BinaryInfC{$X$}
\RightLabel{\footnotesize$\supset$I $(n_3)$}
\UnaryInfC{$A^{*}\supset X$}
\RightLabel{\footnotesize$\supset$E}
\BinaryInfC{$(B^*\supset X)\supset X$}
\AxiomC{$\stackrel{o_1}{B^{*}\supset X}$}
\AxiomC{$\stackrel{m_3}{B^{*}}$}
\RightLabel{\footnotesize$\supset$E}
\BinaryInfC{$X$}
\RightLabel{\footnotesize$\supset$I $(m_3)$}
\UnaryInfC{$B^{*}\supset X$}
\RightLabel{\footnotesize$\supset$E}
\BinaryInfC{$X$}
\RightLabel{\footnotesize$\supset$I $(o_1)$}
\UnaryInfC{$(B^{*}\supset X)\supset X$}
\RightLabel{\footnotesize$\supset$I $(o_2)$}
\UnaryInfC{$A^*\curlyvee B^*$}
\DisplayProof\  =_{\eta}^{2\supset}$$

\medskip
$$=_{\eta}^{2\supset}  \ \AxiomC{$\D D^{*}$}
\noLine
\UnaryInfC{$(A\lor B)^{*}$}
\RightLabel{\footnotesize$\forall E$}
\UnaryInfC{$A^{*}\curlyvee B^{*}$}\DisplayProof \equiv \ \mathscr{D}''^*
$$\endgroup
\end{proof}

\begin{remark}
By inspecting the proof of proposition \ref{etaa} it is clear that for m-closed instances of $=_{\gamma}^{\vee}$, we have that $\mathscr{D}'=_{\gamma^{mc}}^\vee \mathscr{D}''$ implies $\mathscr{D}'^*=_{\beta\eta^{mc}\varepsilon^{mc}}^{2\supset} \mathscr{D}''^*$. Thus, proposition \ref{etaa} together with theorem \ref{naturality-th}  (and with remark \ref{eta2imp-pres}) provides an alternative way to establish corollary \ref{mc-etapres}.
\end{remark}

\subsection{Generalized RP-connectives and $\varepsilon$-equations}
The results of the previous section can be given for the intuitionistic connectives $\wedge$ and $\bot$ as well, by using their RP-translations $\forall X((A\supset B\supset X)\supset X)$ and $\forall X X$. For instance, the generalized permutation \eqref{gammabotG}  for $\bot$ of section \ref{sec-prel} is mapped by the RP-translation onto the following instance of \eqref{epsilonew}:
\begingroup\makeatletter\def\f@size{10}\check@mathfonts 
$$
\AXC{$\mathscr{D}^*$}\noLine\UIC{$\forall X X$}
\UnaryInfC{$A$}\noLine\UIC{$\mathscr{D}_2$}\noLine\UIC{$B$}\DP \  =_{\varepsilon}\
\AXC{$\mathscr{D}^*$}\noLine\UIC{$\forall X X$}
\UnaryInfC{$B$}\DP
$$
\endgroup
Actually, these results can be generalized to the much wider class of connectives introduced by Schroeder-Heister in the context of his natural extension of natural deduction with rules of arbitrary level \cite{SH1984}. 

According to \cite{Pra79}, given $r$ introduction rules for the connective $\dagger$, which have the following general form (for $1\leq h\leq r$):
\begingroup\makeatletter\def\f@size{10}\check@mathfonts 
\begin{equation*}
\AXC{$[A^{h}_{1{1}}]\ \ \ldots\ \ [A^{h}_{k_{1}{1}}]$}\noLine\UIC{$B^{h}_{1}$}\AXC{\ldots}\AXC{$[A^{h}_{1{n_{h}}}] \  \ \ldots \ \ [A^{h}_{k_{n_{h}}{n_{h}}}]$}\noLine\UIC{$B^{h}_{n_{h}}$}\RightLabel{\footnotesize $\dagger I_{h}$}\TIC{$\dagger(C_1\ldots C_{m})$}\DP
\end{equation*}\endgroup
where all the $A^{h}_{lj}$ and $B^{h}_j$ are identical with one of the $C_i$ ($1 \leq l\leq k_{j}$, $1\leq j\leq n_{h}$ and $1 \leq i\leq m$). Applications of the rule can discharge the assumptions of the form $A^{h}_{l{j}}$ ($1 \leq l\leq kj$) in the derivation of the premise $B^{h}_{j}$ ($1<j<n_{h}$).

Given $r$ introduction rules $\dagger I_{1},\dots, \dagger I_{r}$ for the connective $\dagger$ of the form above, a unique elimination rule construed after the model of disjunction which ``inverts'' (in the sense of Lorenzen's inversion principle) this collection of introduction rules is the following:
\begingroup\makeatletter\def\f@size{10}\check@mathfonts 
\begin{equation*}\label{elim}
 \AXC{$\dagger(C_1\ldots C_{m})$}
 \AXC{$[R_{1}^{1}]\ \  \ldots \ \ [R_{n_{1}}^{1}]$}
 \noLine
 \UIC{$X$}
 \AXC{\ldots}
 \AXC{$[R_{1}^{r}]\ \  \ldots \ \ [R_{n_{r}}^{r}]$}
 \noLine
 \UIC{$X$}
 \RightLabel{\footnotesize $\dagger$E}
 \QuaternaryInfC{$X$}
 \DP5%
 \end{equation*}\endgroup
where $X$ is fresh and each $R_{j}^{h}$ discharged by the $\dagger$E rule corresponds to the $j$-th premise of the $\dagger$-intro$_h$ rule. I.e., \mbox{$ R_{j}^{h} = A_{1j}^{h} \supset \dots \supset A_{k_{j}j}^{h}\supset B_{j}^{h}$} if $k>0$; $ R_{j}^{h} =B_{j}^{h}$ otherwise.

The RP-translation of $\dagger (C_1\ldots C_{m})$ is given by the formula 
\begingroup\makeatletter\def\f@size{10}\check@mathfonts 
\begin{equation*}
\dagger(C_1\ldots C_{m})^{*}:=\forall X \big (  (R_{1}^{1}\supset \dots \supset R_{n_{1}}^{1} \supset X) \supset \dots \supset  (R_{1}^{r}\supset \dots \supset R_{n_{r}}^{r} \supset X) \supset X 
\big )
\end{equation*}\endgroup
As all formulas of the form 
$R_{1}^{h}\supset \dots \supset R_{n_{h}}^{h}\supset X$
are p-$X$, the formula $\dagger(C_1\ldots C_{m})^{*}$ is of the form $\forall X F$, for $F$ nested p-$X$, and thus results analogous to proposition \ref{conv} and \ref{etaa} can be established for all such connectives as well. 

In fact, these formulas form a dinstguished sub-class of p-$X$ formulas to which we refer as strictly positive formulas:
\begin{definition}[strictly positive and nested strictly positive formulas]
If $C$ is p-$X$ and either $X$ does not occur free in $C$ at all, or it occurs free only once, as the rightmost variable in $C$, we call $C$ \emph{strictly positive in $X$}, abbreviated sp-$X$. 

More formally, let $ \forall\OV Y_{i}$ (for $1\leq i\leq n$) denote finite (possibly empty) sequences of quantifiers. Then $C$ is sp-$X$ when  $C$ is $$\forall \OV Y_{1}(F_{1}\supset \forall \OV Y_{2}(F_{2}\supset \dots \supset \forall \OV Y_{n}(F_{n}\supset Z)\dots ))$$ where $X$ does not occur free in any of the $F_i$ (for $1\leq i\leq n$).
 
Let $ \forall\OV Y_{i}$ (for $1\leq i\leq n$) denote finite (possibly empty) sequences of quantifiers. A formula of the form 
\begingroup\makeatletter\def\f@size{10}\check@mathfonts 
 $$\forall \OV Y_{1}(F_{1}\supset \forall \OV Y_{2}(F_{2}\supset \dots \supset \forall \OV Y_{n}(F_{n}\supset X)\dots ))$$\endgroup
 is {\em nested sp-$X$} provided that $\forall \OV Y_1 F_{1},\dots, \forall \OV Y_1 \ldots \forall \OV Y_n F_{n}$ are sp-$X$  formulas.
 \end{definition}
 
 The RP-translation of $\dagger(C_1\ldots C_{m})^{*}$ is in fact of the form $\forall X F$, for $F$ nested sp-$X$.

In \cite{SH1984}, Schroeder-Heister got rid of the left-iterated implications in the elimination rules by enriching the structural means of expression of natural deduction  by allowing not only formulas but also (applications of) rules to be assumed in the course of a derivation and by allowing rules to discharge not only formulas but also previously assumed rules. Once the structural device of rule-discharge is introduced, nothing prohibits its use in introduction rules, thereby yielding a yet richer class of connectives definable by means of introduction and elimination rules. In \cite{SH14hlr}, the structural means of expression for defining connectives have been further enriched by admitting a form of structural quantification, in terms of which, for instance, the introduction rule for negation can be formulated as: 

\begingroup\makeatletter\def\f@size{10}\check@mathfonts 
\begin{lrbox}{\mypti}% Store prooftree in \mypti
\begin{varwidth}{\linewidth}
$\left(\AXC{$[A]$}\noLine\UIC{$X$}\DP\right)_X$
\end{varwidth}
\end{lrbox}
$$\AXC{\usebox{\mypti}}\UIC{$\neg A$}\DP$$\endgroup	
where the notation $()_X$ indicates that $X$ plays a role similar to the {\em eigenvariable} $X$ in the second order $\forall$I. As the elimination rule obtained by inverting a collection of introduction rules of this more general form follows the same pattern above, it should be clear that the RP-translation can be generalized accordingly, and that the RP-translation of a logically complex formula governed by such a connective will be of the form $\forall X   F$, with $F$ sp-$X$. 

As  Schroeder-Heister \cite{SH14hlr} has shown, however, one can define (well-behaved) connectives by specifying introduction and elimination rules for them that do not follow the strict pattern given above, but a more relaxed condition (but see \cite{Tra17} for a criticism of Schroeder-Heister's proposal). As natural transformations between sp-$X$ formulas are enough to describe permutations for connectives obeying inversion, we conjecture that natural transformations between formulas containing only positive (not necessarily strictly positive) occurrences of $X$ can be used to represent permutations for a more general class of connectives.  It seems reasonable to conjecture also the converse, namely that permutative conversions {\em cannot} be defined for those connectives whose RP-translation is not a nested p-$X$ formula. Further investigations of these matters will be undertook in a subsequent paper.

\section{Conclusions}

In this paper we have introduced pn-$X$ derivations and shown that they satisfy a naturality condition which translates a restricted version of the dinaturality condition in functorial semantics. Moreover we showed the exact sense in which the naturality of p-$X$ derivations corresponds to a principle of permutation which holds for closed derivations of quantified nested p-$X$ formulas. 

The RP-translation suggested to consider a particular sub-class of nested p-$X$ formulas, nested sp-$X$ formulas, the universal closure of which translate the formula constructed with propositional connectives satisfying the inversion principle of Prawitz and Schroeder-Heister: The permutative principle for nested sp-X formulas translates the generalized permutative and eta conversions for these connectives. 

In recent work Ferreira and Ferreira \cite{Ferreira2009} have established similar results by embedding intuitionistic propositional logic into the fragment of \Nd in which $\forall$E is restricted to atomic instantiation. In a follow-up paper we will investigate the relationship between the their approach and ours and investigate their generalization to the case of nested p-$X$ formulas in relation to the connectives definable using rules of higher level and propositional quantification.

%%%%%%  ACKNOWLEDGEMENTS SECTION
\paragraph{Acknowledgements.} 
This work has been carried out as part of the ANR-DFG project ``Beyond Logic'' (Schr275/17-1 / ANR-14-FRAL-0002) The first author acknowledges moreover the support of DFG who financed his work as part of the project ``Logical Consequence and Paradoxical Reasoning'' (Tr1112/3-1).

\appendix

\section{Proof of proposition \ref{normalspx}}\label{appendino}

\begin{definition}[Sub-formula]
The {\em sub-formulas} of $A$ are defined by induction on the number of logical signs in $A$ as follows:
\begin{itemize}
 \item if $A\equiv X$, the only sub-formula of $A$ is $A$ itself;
 \item if $A\equiv B\supset C$, the sub-formulas of $A$ are $A$ itself and the sub-formulas of $B$ and  $C$;
 \item if $A\equiv \forall X B$, the sub-formulas of $A$ are $A$ itself and the sub-formulas of $B$;
\end{itemize}
\end{definition}

\begin{lemma}\label{spix}
If $A$ is pn-$X$ (resp. sp-$X$), all its sub-formulas are pn-$X$ (resp. sp-$X$).%If $\mathscr{D}$ is a normal sp-$X$ derivation, all formulas occurring in $\mathscr{D}$ are sp-$X$. \LUCA{non trivial, perche non c'e sottoformula in sistema F, mi pare. In ogni caso, ci serve sto remark?}.
\end{lemma}

Lemma \ref{spix} is not enough to warrant that in a $\beta$-normal pn-$X$ (resp. sp-$X$) derivation all formulas are pn-$X$ (resp. sp-$X$), since normal derivations in \Nd do not enjoy the sub-formula property. However, we can show that if a pn-$X$ (resp. sp-$X$) derivation is $X$-safe, then this is the case, as a consequence of a weakened form of the sub-formula property (see proposition \ref{subf}).

\begin{definition}[$X$-equivalence]  We say that $A<_{X}' B$ iff for some $Y\nequiv X$ and for some $C$, such that $X$ does not occur in $C$, $B\equiv A\llbracket C/Y\rrbracket$. 

Let the relation $A<_{X} B$ and $='_{X}$ be defined (respectively) as the transitive closure and the reflexive and symmetric closure  of the relation $A<_{X}' B$.%, given by
%$$A\leq_{X}^{1} A\llbracket C/Z\rrbracket  \qquad \qquad Z\neq X, X \text{ not occurring in }C$$
%Let moreover $=_{X}$ be the reflexive and symmetric closure of $\leq_{X}$.

Finally, let the relation $=_{X}$ be the union of  $A<_{X} B$ and $='_{X}$ (i.e.~the reflexive symmetric and transitive closure of $A<_{X}' B$). 
\end{definition}

To establish the weakened sub-formula property for pn-$X$ and sp-$X$ derivations we first prove the following:

\begin{lemma}\label{jijji1} If $A<_X B\supset C$, then there exist sub-formulas $A_{1}$ and $A_{2}$ of $A$ such that $A_{1}\leq_{X} B$ and $A_{2}\leq_{X}C$;
%\item if $A\leq_{X} \forall Z B$, then there exists a sub-formula $A'$ of $A$ such that $A\leq_{X} B$.
\end{lemma}
\begin{proof}We prove the two parts of the lemma separately:
Let $A<_XB\supset C$. Then there exists a finite sequence of substitutions (whose image is made of formulas not containing $X$) $\theta_{1},\dots, \theta_{n}$ such that $A\theta_{1}\dots \theta_{n}\equiv B\supset C$. If $A$ is of the form $A_{1}\supset A_{2}$, then $A\theta_{1}\dots \theta_{n}\equiv A_{1}\theta_{1}\dots \theta_{n}\supset A_{2}\theta_{1}\dots \theta_{n}$, which proves the claim. Otherwise, since $A$ cannot be of the form $\forall Z A'$ (as $(\forall Z A')\theta_{1}\dots \theta_{n}\equiv \forall Z (A\theta'_{1}\dots \theta'_{n})$, where $\theta_{i}'$ is obtained by renaming all occurrences of $Y$), $A$ must be a variable $Y$, and then $A\llbracket B/Y\rrbracket \equiv B$ and $A\llbracket C/Y\rrbracket \equiv C$, so we can take $A_{1}\equiv A_{2}\equiv A$.

%Let now $A<_X\forall ZB$. Again, there exists a finite sequence of substitutions $\theta_{1},\dots, \theta_{n}$ (whose image is made of formulas not containing $X$) such that $A\theta_{1}\dots \theta_{n}= \forall ZB$. If $A$ is of the form $\forall Z A'$, then $A\theta_{1}\dots \theta_{n}= \forall Z (A\theta'_{1}\dots \theta'_{n})$, hence $A\theta'_{1}\dots \theta'_{n}=B$. Otherwise, since $A$ cannot be of the form $A_{1}\supset A_{2}$ (as $(A_{1}\supset A_{2})\theta_{1}\dots \theta_{n}=A_{1}\theta_{1}\dots \theta_{n}\supset A_{2}\theta_{1}\dots \theta_{n}$), $A$ must be a variable $Y$, and then $A[\forall Z B/Y]=\forall Z B$, so we can take $A'=A$.

\end{proof} 

\begin{lemma}\label{jijji2}
If $A=_{X}B$ and $A$ is pn-$X$ (resp. sp-$X$), then $B$ is pn-$X$ (resp. sp-$X$).
\end{lemma}
\begin{proof}
 
By induction on $A$ one shows that, if $C$ does not contain occurrences of $X$, then $A$ is pn-$X$ (resp. sp-$X$) if and only if $A\llbracket C/Y\rrbracket$ (for $Y\nequiv X$) is pn-$X$ (resp. sp-$X$). This proves the claim for $=_{X}'$ (the reflexive and symmetric closure of $<_X'$). The claim can then be extended to $=_{X}$ by induction on the application to $A$ of a finite number of substitutions.
\end{proof}

We can now establish the following weakened form of the sub-formula property for $\beta$-normal, $X$-safe and sp-$X$ derivations:

\begin{proposition}\label{subf}
Let $\D D$ be a $\beta$-normal $X$-safe derivation. Then, for any formula $F$ occurring in $\D D$, 
\begin{description}
\item[$i.$] for some formula $C$, which is a sub-formula either of the conclusion of $\D D$, or of some undischarged assumption of $\D D$, $F=_{X} C$;
\item[$ii.$] moreover, if $\D D$ ends by an elimination rule, then it has a \emph{principal branch}, i.e. a sequence of formulas $A_{0},\dots, A_{n}$ such that
	\begin{itemize}
	\item $A_{0}$ is an undischarged assumption of $\D D$;
	\item $A_{n}$ is the conclusion of $\D D$;
	\item for all $1\leq i\leq n-1$, $A_{i}$ is the major premise of an elimination rule whose consequence is $A_{i+1}$.

	\end{itemize}

\end{description}
\end{proposition} 
\begin{proof}
We argue by induction on $\D D$. If $\D D$ consists solely of an assumption, there is nothing to prove. If $\D D$ ends with an application of an introduction rule, we must consider two cases:

\begin{enumerate}
\item $\D D$ ends with an application of $\supset$I:
\begingroup\makeatletter\def\f@size{10}\check@mathfonts
$$\AXC{$[\stackrel{n}{A}]$}\noLine\UIC{$\D D'$}\noLine\UIC{$B$}\RightLabel{$\supset$I $(n)$}\UIC{$A\supset B$}\DP$$\endgroup
Let $F$ be a formula occurring in $\D D$. Unless $F$ is $A\supset B$, in which case there is nothing to prove, $F$ occurs in $\D D'$. Then, by induction hypothesis, two possibilities arise: either for some sub-formula $C$ of an undischarged assumption $C'$ of $\D D'$, $F=_{X}C$; or for some sub-formula $C$ of $B$, $F=_{X}C$. In the first case, if $C'$ is different from $A$, then it is an undischarged assumption of $\D D$ and we are done; if $C'$ is $A$, we conclude by remarking that $A$ is a sub-formula of the conclusion $A\supset B$ of $\D D$. In the second case, we conclude similarly by remarking that $B$ is a sub-formula of the conclusion $A\supset B$.
\item $\D D$ ends with an application of $\forall$I rule:
\begingroup\makeatletter\def\f@size{10}\check@mathfonts
$$\AXC{$\D D'$}\noLine\UIC{$A$}\RightLabel{$\forall$I}\UIC{$\forall Y A$}\DP$$\endgroup
Let $F$ be a formula occurring in $\D D$. Again, unless $F$ is $\forall YA$, in which case there is nothing to prove, $F$ occurs in $\D D'$. Then, by induction hypothesis, either for some sub-formula $C$ of an undischarged assumption $C'$ of $\D D'$, $F=_{X}C$, or for some sub-formula $C$ of $A$, $F=_{X}C$. In the first case, we are done, as all undischarged assumptions of $\D D'$ are undischarged assumptions of $\D D$; in the second case, we conclude by remarking that $A$ is a sub-formula of the conclusion $\forall YA $ of $\D D$.
\end{enumerate}

If $\D D$ ends with an application of an elimination rule, again we must consider two cases:
\begin{enumerate}
\item $\D D$ ends with an application of $\supset$E:
\begingroup\makeatletter\def\f@size{10}\check@mathfonts
$$\AXC{$\D D_{1}$}\noLine\UIC{$A\supset B$}\AXC{$\D D_{2}$}\noLine\UIC{$A$}\RightLabel{$\supset$E}\BIC{$B$}\DP$$\endgroup
As $\D D$ is $\beta$-normal, the subderivation $\D D_{1}$ cannot end with an introduction rule. Hence, by induction hypothesis, there exists a principal branch $A_{0},\dots, A_{n}\equiv A\supset B$ in $\D D_{1}$. It can be easily shown by induction that, for all $1\leq i\leq n$, there exists a sub-formula $A'$ of $A_{0}$ such that $A'<_XA_{i}$. Hence, in particular, there exists a sub-formula $A'$ of $A_{0}$ such that $A'<_X A\supset B$. 

Now, if $F$ occurs in $\D D$ then, unless $F$ is $B$, in which case we are done, $F$ must occur either in $\D D_{1}$, either in $\D D_{2}$. In the first case, by induction hypothesis, either for some $C$ sub-formula of an undischarged assumption of $\D D_{1}$, $F=_{X} C$, in which case we are done, or for some sub-formula $C$ of $A\supset B$, $F=_{X}C$, in which case we use the fact that $A\supset B=_{X}A'$ and the transitivity of $=_{X}$. In the second case, either for some $C$ sub-formula of an undischarged assumption of $\D D_{2}$, $F=_{X}C$, in which case we are done, or for some sub-formula of $A$, $F=_{X}C$; in this last case, as $A$ is a sub-formula of $A\supset B$, by lemma \ref{jijji1}, there exists a sub-formula $A''$ of $A'$ (and hence of $A_{0}$), such that $A''<_XA$.

Finally, take $A_{n+1}\equiv B$ and we obtain a principal branch for $\D D$.

\item $\D D$ ends with an application of $\forall$E:
\begingroup\makeatletter\def\f@size{10}\check@mathfonts
$$\AXC{$\D D'$}\noLine\UIC{$\forall ZA$}\RightLabel{$\forall$E}\UIC{$A\llbracket B/Z\rrbracket $}\DP$$\endgroup
As $\D D$ is $X$-safe, $X$ does not occur in $B$. Moreover, as $\D D$ is $\beta$-normal, by proposition \ref{closenormintro} the subderivation $\D D'$ cannot end by an introduction. Hence, by induction hypothesis, there exists a principal branch $A_{0},\dots, A_{n}=\forall ZA$ in $\D D'$. Then, there exists a sub-formula $A'$ of $A_{0}$ such that $A'<_X\forall ZA$.

Now, if $F$ occurs in $\D D$ then, unless $F\equiv A\llbracket B/Z\rrbracket $, in which case we are done, $F$ must occur in $\D D'$. Then, by induction hypothesis, either there exists a sub-formula $C$ of an undischarged assumption of $\D D'$ such that $F=_{X}C$, in which case we are done, or for some sub-formula $C$ of $\forall Z A$, $F=_{X}C$, in which case we use the fact that $A'=_{X}\forall ZA$ as well as the transitivity of $=_{X}$.

Finally, take $A_{n+1}\equiv A\llbracket B/Z\rrbracket$ and we obtain a principal branch for $\D D$.
\end{enumerate}
\end{proof}

Proposition \ref{normalspx} is an immediate corollary of proposition \ref{subf}:
\begin{proof}[Proof of proposition \ref{normalspx}.]
By proposition \ref{subf}, if a formula $F$ occurs in $\D D$, then for some formula $C$, which is either a sub-formula of an undischarged assumption of $\D D$, either a sub-formula of the conclusion of $\D D$, $F=_{X}C$. By lemma \ref{spix}, $C$ is sp-$X$. Hence, by lemma \ref{jijji2}, $F$ must be pn-$X$ \mbox{(resp.~sp-$X$)}.
\end{proof}

%%%%%%  AUTHOR'S ADDRESS INFORMATION AT THE END OF THE PAPER:
%
%\AuthorAdressEmail{Luca Tranchini}{Department of Computer Science\\
%Eberhard Karls Universit\"at T\"ubingen\\
%Sand 13 \\
%72076 T\"ubingen, Germany}{luca.tranchini@gmail.com}
%
%%%%  In case of more than one author use also 
%%%%  the following as many times as needed:
%
%\AdditionalAuthorAddressEmail{Paolo Pistone}{Dipartimento di Matematica e Fisica\\
%Universit\`a Roma Tre\\
%Largo S. Leonardo Murialdo, 1\\
%00146, Rome, Italy}{paolo.pistone@uniroma3.it}
%
%
%\AdditionalAuthorAddressEmail{Mattia Petrolo}{IHPST, CNRS, ENS\\
%Universit\'e Paris 1 Panth\'eon Sorbonne\\
%13, rue du Four\\
%75006 Paris, France}{mattia.petrolo@univ-paris1.fr }
%\end{document}
%

%%%%%%  END  

\end{document}